\documentclass[12pt,reqno]{amsart}
\usepackage{amsmath,amssymb,latexsym}
\usepackage{amstext}
\usepackage{extarrows}
\usepackage{enumitem}
\usepackage{a4wide}
\usepackage{graphicx}
\usepackage{appendix}
\allowdisplaybreaks \numberwithin{equation}{section}
\usepackage{color}

\usepackage[colorlinks,urlcolor=blue]{hyperref}

\numberwithin{equation}{section}

\newtheorem{theorem}{Theorem}[section]
\newtheorem{proposition}[theorem]{Proposition}
\newtheorem{corollary}[theorem]{Corollary}
\newtheorem{lemma}[theorem]{Lemma}

\theoremstyle{definition}

\newtheorem{definition}[theorem]{Definition}

\theoremstyle{remark}
\newtheorem{remark}[theorem]{Remark}

\newcommand{\hsgm}{\mathcal{H}^{\sigma}}
\newcommand{\veps}{\varepsilon}
\newcommand{\dlt}{\delta}
\newcommand{\ptl}{\partial}
\newcommand{\bk}{\mathbf{k}}
\newcommand{\gk}{\mathbf{K}}
\newcommand{\bm}{\mathbf{m}}
\newcommand{\bn}{\mathbf{n}}

\newcommand{\bv}{\mathbf{v}}
\newcommand{\br}{\mathbb{R}}
\newcommand{\bx}{\boldsymbol{x}}
\newcommand{\by}{\boldsymbol{y}}
\newcommand{\gz}{\boldsymbol{z}}
\newcommand{\bz}{\mathbb{Z}}
\newcommand{\balpha}{\boldsymbol{\alpha}}

\newcommand{\bxi}{\boldsymbol{\xi}}
\newcommand{\bzeta}{\boldsymbol{\zeta}}
\newcommand{\bla}{\big\langle}
\newcommand{\bra}{\big\rangle}
\newcommand{\bav}{\Big|}

\newcommand{\bkappa}{\boldsymbol{\kappa}}
\newcommand{\ml}{\mathcal{L}}
\newcommand{\bp}{\boldsymbol{p}^{\sigma}}
\newcommand{\vbl}{\mathcal{L}_{\perp}^{^{-1}}}

\begin{document}

\title[Wave packets in fNLS with a honeycomb potential]
{Wave packets in the fractional nonlinear Schr\"odinger equation with a honeycomb potential}

\author{Peng Xie}
\author{Yi Zhu}

\address{Zhou Pei-Yuan Center for Applied Mathematics, Tsinghua University, Beijing 10084, P.R. China}
\email{xiep14@mails.tsinghua.edu.cn}
\address{Yau Mathematical Sciences Center and Department of Mathematical Sciences, Tsinghua University, Beijing 10084, P.R. China}
\email{yizhu@tsinghua.edu.cn}

\subjclass[2010]{35Q41, 35Q60, 35C20, 35R11, 35P05}

\date{\today}

\keywords{Fractional Schr\"odinger equation, Honeycomb structure, Effective dynamics}

\maketitle


\begin{abstract}
In this article, we study wave dynamics in the fractional nonlinear Schr\"odinger equation with a modulated honeycomb potential. This problem arises from recent research interests in the interplay between topological materials and nonlocal governing equations. Both are current focuses in scientific research fields. We first develop the Floquet-Bloch spectral theory of the linear fractional Schr\"odinger operator with a honeycomb potential. Especially, we prove the existence of conical degenerate points, i.e., Dirac points, at which two dispersion band functions intersect. We then investigate the dynamics of wave packets spectrally localized at a Dirac point and derive the leading effective envelope equation. It turns out the envelope can be described by a nonlinear Dirac equation with a varying mass. With rigorous error estimates, we demonstrate that the asymptotic solution based on the effective envelope equation approximates the true solution well in the weighted-$H^s$ space.
\end{abstract}

\section{Introduction}

This work is concerned with the following fractional nonlinear Schr\"odinger equation (fNLS) with a modulated honeycomb potential
\begin{equation}\label{eq:fnls}
\text{i}\veps\ptl_t\psi =(-\veps^2\Delta_{\bx})^{\frac{\sigma}{2}}\psi+V(\frac{\bx}{\veps})\psi+\veps\kappa(\bx)W(\frac{\bx}{\veps})\psi+\veps\mu|\psi|^2\psi,\quad (t\in \br^+,~\bx\in\br^2),
\end{equation}
where $\psi=\psi(t, \bx)$ represents the wave field, $1<\sigma\leqslant2$ is the fractional parameter, $\mu=\pm1$ stands for the focusing or defocusing coefficient, $V(\cdot),~W(\cdot)\in C^{\infty}(\br^2,\br)$ are periodic potentials, and $\kappa(\cdot)\in C^{\infty}(\br^2,\br)$ is the bounded modulation. The small parameter $0<\veps\ll 1$ describes the ratio between microscopic and macroscopic scales. The definitions of fractional Laplacian and other assumptions are given in Section \ref{sec:preli}. It is well known that the nonlinear Schr\"odinger equation is the standard model in many wave systems and has been widely studied in different aspects \cite{sulem2007nonlinear}. The equation under consideration in this work has two new ingredients---fractional Laplacian and honeycomb potential with a spacially varying modulation. Both terminologies are current focuses due to recent advances in nonlinear optics, topological quantum mechanics, meta-materials \cite{ammari2018honeycomb,ammari2018high,laskin2000fractionallevy,longhi2015fractional,rechtsman2013Photonic,zhang2015propagation}.

From the application point of view, a key problem is to understand the wave dynamics governed by \eqref{eq:fnls}. The aim of our current work is to derive and justify the effective envelope equation which only has the macroscopic scale with the high oscillations being homogenized. This is a very efficient and useful treatment to understand complicated wave dynamics. To this end, we make the following achievements. We first develop the spectral theory of the fractional Schr\"odinger operator with a honeycomb potential. Using the Fourier series to define the fractional Laplacian on a bounded domain with quasi-periodic boundary conditions, we find that the Floquet-Bloch theory still applies for the fractional Sch\"odinger operator. Considerably modifying the strategies developed by Fefferman and Weinstein for the standard Schr\"odinger operator with a honeycomb potential \cite{fefferman2012honeycomb}, we prove that the honeycomb structured fractional Schr\"odinger operator $\hsgm=(-\Delta)^{\frac{\sigma}{2}}+V(\cdot)$ has two conically degenerate points, a.k.a, Dirac points, at $\mathbf{K}$ and $\mathbf{K}'$ points, see Theorem \ref{thm:diracpoint} in Section \ref{sec:dirac}.

Then we analyze the wave packets spectrally localized at a Dirac point with the envelope scale matches the spacial modification and nonlinearity. The leading macroscopic envelope is governed by a nonlinear Dirac equation with a varying mass coming from the spacial modulation to the honeycomb potential.  The mathematical justification of this derivation is given by rigorous error estimates in weighted-$H^s$ space in Section \ref{sec:main} for $0<\veps\ll 1$. More specifically, Theorem \ref{thm:leading} and \ref{thm:main} show that the solution $\psi^{\veps}$ to \eqref{eq:fnls} together with initial data $\psi_0^{\veps}$ are spectrally localized at the Dirac point, and can be approximated by the leading term
\begin{equation}\label{id:approxsolu}
\psi^{\veps}(t, \bx)\sim e^{-\mathrm{i}E_D \frac{t}{\veps}}\big(\alpha_1(t,\bx)\Phi_1(\frac{\bx}{\veps})+\alpha_2(t,\bx)\Phi_2(\frac{\bx}{\veps})\big),\quad \veps\rightarrow 0,
\end{equation}
where $E_D$ is the degenerate eigenvalue, and  the amplitudes $\alpha_1(t,\bx),~\alpha_2(t,\bx)$ solve the following nonlinear Dirac equation with a varying mass:
\begin{equation}\label{eq:diracalpha}
\left\{
\begin{aligned}
\ptl_t\alpha_1+v_F^{\sigma}(\ptl_{x_1}+\text{i}\ptl_{x_2})\alpha_2+\text{i}\vartheta\kappa(\bx)\alpha_1+\text{i}\mu (b_1|\alpha_1|^2+ b_2|\alpha_2|^2)\alpha_1&=0  \\
\ptl_t\alpha_2+v_F^{\sigma}(\ptl_{x_1}-\text{i}\ptl_{x_2})\alpha_1-\text{i}\vartheta\kappa(\bx)\alpha_2+\text{i}\mu( b_2|\alpha_1|^2+ b_1|\alpha_2|^2)\alpha_2&=0
\end{aligned}~,
\right.
\end{equation}
with initial datum $\alpha_j(0,\bx)=\alpha_{j_0}(\bx)\in \mathcal{S}(\br^2),~j=1,~2$. Here $v_F^{\sigma},~b_1,~b_2,~\vartheta$ are given real-valued constants.

Mathematically, this work is related to the studies of semi-classical solutions or WKB solutions to dispersive wave systems with periodic coefficients \cite{ablowitz2013Nonlinear2,allaire2011diffractive,arbunich2018rigorous,e2013asymptotic,jin2011mathematical,pelinovsky2011localization}. Regardless of the important insights of this work in the applied fields, the mathematical challenges include dealing with the fractional Laplacian $(-\Delta)^{\frac{\sigma}{2}}$ and the spectral degeneracy caused by the honeycomb symmetry. In the literature, both the fractional Laplacian and honeycomb potentials have been considerably investigated. For example, the fractional Schr\"odinger equation was first proposed by N. Laskin \cite{laskin2000fractionallevy} in quantum mechanism, and then it was found to be useful in optics \cite{longhi2015fractional,zhang2015propagation}. This model has some significant difference from its standard counterpart \cite{caffarelli2007extension,davila2014concentrating,hong2017new}. On the other hand, two-dimensional honeycomb materials possess subtle physical properties and broad prospect of applications, and it turns to be one of most successful example to understand and realize the topological phenomena \cite{ammari2018honeycomb,ammari2018high,bal2019continuous,geim2007The,guo2019bloch,haldane2008Possible,joannopoulos2011photonic,longhi2015fractional,raghu2008analogs}. The most interesting characterization of this structure is the existence of the conical degenerate spectral point lying the dispersion surfaces. Fefferman, Weinstein and their collaborators proved the existence of Dirac points lying on the dispersion surfaces of the honeycomb latticed standard Sch\"odinger and divergence elliptic operators via Lyapunov-Schmidt reduction strategy \cite{fefferman2017topologically,fefferman2012honeycomb,keller2018spectral,lee2019elliptic}. Then, the time evolution of linear and nonlinear wave packet propagation spectrally concentrated around the Dirac point has been rigorously studied, and the effective dynamics is governed by corresponding Dirac equations \cite{ablowitz2010Evolution,ablowitz2013Nonlinear2,hu2019linear,arbunich2018rigorous,fefferman2014wave,xie2019wave}. Merging the two terminologies together in a semi-classical nonlinear evolution equation, we need to deal with mathematical challenges. Examples include formulations of fractional derivatives acting on quasi-periodic functions, asymptotics of the fractional Laplacian acting on a wave packet with highly oscillating Floquet-Bloch modes, homogenization of such modes with degenerate eigenvalues and so on. The results also shed some light on the rigorous analysis of topologically protected wave propagation in honeycomb-based media if additional assumptions are added to the slowly varying modulations \cite{drouot2020edge,fefferman2017topologically,hu2019linear,lee2019elliptic}.

This paper will be organized as follows: In section \ref{sec:preli} and \ref{sec:dirac}, we briefly review the Floquet-Bloch theory from honeycomb latticed fractional Sch\"odinger operator and verify the existence of Dirac point. From the \ref{sec:main}th section, we start to derive the effective dynamics of wave packet problem by fractional nonlinear Schr\"odinger equation. We put the rigorous proof in section \ref{sec:proof} and apply a micro-scaled Bloch decomposition method to derive the macro problem by rescaling.

\section{Preliminaries}\label{sec:preli}

\subsection{Floquet-Bloch theory for fractional Schr\"odinger operator}

In this subsection, we introduce a brief description to the Floquet-Bloch theory for the fractional Schr\"odinger operator with a periodic potential \cite{bensoussan1978asymptotic,fefferman2012honeycomb,kuchment2012floquet,roncal2016fractional}.

The hexagonal lattice $\Lambda$ is generated by two linear independent vectors $\bv_1~\text{and}~\bv_2$, i.e.,
\begin{equation}
\Lambda=\mathbb{Z}\bv_1\oplus\mathbb{Z}\bv_2,~\quad\bv_1=
\begin{pmatrix}
\frac{\sqrt{3}}{2} \\
\frac{1}2
\end{pmatrix},\quad
\bv_2=
\begin{pmatrix}
\frac{\sqrt{3}}{2} \\
-\frac1{2}
\end{pmatrix},
\end{equation}
and the fundamental cell $\Omega= \{\theta_1\bv_1+\theta_2\bv_2: 0\leqslant\theta_j\leqslant 1,~ j=1,~ 2\}$. The corresponding dual lattice and fundamental cell are $\Lambda^*=\mathbb{Z}\bk_1\oplus\mathbb{Z}\bk_2$ and $\Omega^*= \{\theta_1\bk_1+\theta_2\bk_2: -1/2\leqslant\theta_j\leqslant 1/2,~ j=1,~ 2\}$, where two dual vectors $\bk_1,~\bk_2$ satisfy $\bv_i\cdot\bk_j=2\pi\dlt_{ij}$.

We introduce the following function spaces:
\begin{eqnarray}
&L^2(\br^2/\Lambda)=\left\{\phi(\by)\in L^2_{\text{loc}}\left(\mathbb{R}^2,\mathbb{C}\right):~\phi(\by+\bv)=\phi(\by),~ \forall~ \bv \in \Lambda,~ \by \in \mathbb{R}^2 \right\}, &\\
&\text{and}\quad L^2_{\bk}(\br^2/\Lambda)=\{\Phi(\by):\  e^{-\mathrm{i}\bk\cdot \by}\Phi(\by)\in L^2(\br^2/\Lambda) \}.&
\end{eqnarray}
Note that functions in $L^2_{\bk}(\br^2/\Lambda)$ are quasi-periodic. Namely, if $\Phi(\by) \in L^2_{\bk}(\br^2/\Lambda)$, then $\Phi(\by+\bv)=e^{\mathrm{i}\bk\cdot\bv}\Phi(\by), \forall~ \bv \in \Lambda$.
Similarly, we can also define $H_{\bk}^s(\br^2/\Lambda)$ and $C_{\bk}^{\infty}(\br^2/\Lambda)$ in a standard way.

The standard Fourier expansion of $\phi\in L^2(\br^2/\Lambda)$ is given as follows
\begin{equation}
	\phi(\by)=\sum_{\bm\in\mathbb{Z}^2}\hat{\phi}(\bm\vec{\bk}) e^{\text{i}\bm\vec{\bk}\cdot\by},\quad \hat{\phi}(\bm\vec{\bk})=\frac{1}{|\Omega|}\int_{\Omega}e^{-\text{i}\bm\vec{\bk}\cdot\by}\phi(\by)d\by.
\end{equation}
where $\bm=(m_1,m_2)\in \mathbb{Z}^2$ and $\bm\vec{\bk}=m_1\bk_1+m_2\bk_2$ for notational convenience. Now $\forall~\bk\in\Omega^*$, the Fourier series of $\Phi\in L^2_{\bk}(\br^2/\Lambda)$ is
\begin{equation}
	\Phi(\by)=\sum_{\bm\in\mathbb{Z}^2}\hat{\Phi}(\bk+\bm\vec{\bk})e^{\text{i}(\bk+\bm\vec{\bk})\cdot\by},\quad \hat{\Phi}(\bk+\bm\vec{\bk})=\frac{1}{|\Omega|}\int_{\Omega}e^{-\text{i}(\bk+\bm\vec{\bk})\cdot\by}\Phi(\by)d\by.
\end{equation}

As it will be seen later, we need to deal with the fractional Laplacian derivative on functions in $H_{\bk}^{\sigma}(\br^2/\Lambda)$. In this work, it is natural to introduce the following definition, for any $\Phi\in H_{\bk}^{\sigma}(\br^2/\Lambda)$
\begin{equation}\label{id:delta}
	(-\Delta)^{\frac{\sigma}{2}}\Phi(\by)=\sum_{\bm\in\mathbb{Z}^2}|\bk+\bm\vec{\bk}|^{\sigma}\hat{\Phi}(\bk+\bm\vec{\bk})e^{\text{i}(\bk+\bm\vec{\bk})\cdot\by}=\sum_{\bm\in\mathbb{Z}^2}|\bk+\bm\vec{\bk}|^{\sigma}\hat{\phi}(\bm\vec{\bk})e^{\text{i}(\bk+\bm\vec{\bk})\cdot\by},
\end{equation}
where $\phi(\by)=e^{-\text{i}\bk\cdot\by}\Phi(\by)$.

Meanwhile, if $f\in H^{\sigma}(\br^2)$, we also introduce the fractional Laplacian in the sense of Fourier transform on $\br^2$ raised by \cite{caffarelli2007extension,laskin2000fractionallevy}.
\begin{equation}\label{laplacian:whole}
	(-\Delta)^{\frac{\sigma}2}f(\bx)=\int_{\br^2}|\bzeta|^{\sigma}~\widehat{f}(\bzeta) e^{\mathrm{i}\bzeta\cdot\bx}d\bzeta,\quad \widehat{f}(\bzeta)=\frac{1}{4\pi^2}\int_{\br^2}e^{-\mathrm{i}\bzeta\cdot\bx}f(\bx)d\bx.
\end{equation}
By Plancherel theory, it is also equivalent to the following singular integral,
\begin{align}
	(-\Delta)^{\frac{\sigma}2}f(\bx)=C_{2,\sigma}\text{P.V.}\int_{\br^2}\frac{f(\bx)-f(\gz)}{|\bx-\gz|^{2+\sigma}} d\gz.
\end{align}
P.V. represents the $Principal~Value$, and $C_{2,\sigma}$ is a given normalized constant \cite{caffarelli2007extension,lemm2016holder,secchi2013ground}.

With the above definition of fractional Laplacian, the Floquet-Bloch theory applies. Assume that the potential $V(\by)$ is smooth, real-valued and $\Lambda-$periodic, then $\hsgm=(-\Delta)^{\frac{\sigma}{2}}+V(\by)$ is self-adjoint. $\forall~\bk\in\Omega^*$, consider the following eigenvalue problem,
\begin{equation}\label{id:eigenprob:pseodu}
\hsgm\Phi(\by;\bk)=E(\bk)\Phi(\by;\bk),\quad \Phi(\by;\bk)\in L^2_{\bk}(\br^2/\Lambda).
\end{equation}
Alternatively, let $\phi(\by;\bk)=e^{-\text{i}\bk\cdot\by}\Phi(\by;\bk)\in L^2(\br^2/\Lambda)$ and $\hsgm(\bk)=e^{-\text{i}\bk\cdot\by}\hsgm e^{\text{i}\bk\cdot\by}$, then the eigenvalue problem on the torus turns into
\begin{align}\label{id:eigenprob:periodic}
	\hsgm(\bk)\phi(\by;\bk)=\sum_{\bm\in\mathbb{Z}^2}|\bk+\bm\vec{\bk}|^{\sigma}\hat{\phi}(\bm\vec{\bk};\bk)e^{\text{i}\bm\vec{\bk}\cdot\by}+V(\by)\phi(\by;\bk)=E(\bk)\phi(\by;\bk).
\end{align}

According to Hilbert-Schimdt theorem for elliptic operators \cite{reed1978methods}, we can develop the Floquet-Bloch theory for the fractional Schr\"odinger operator with quasi-periodic boundary eigenvalue problem in \eqref{id:eigenprob:pseodu}. Namely, the following statements hold
\begin{enumerate}
	\item There exists an ordered eigenvalue series for each $\bk\in\Omega^*$,
	\begin{equation}\label{eigenv:order}
		E_1(\bk)\leqslant E_2(\bk)\leqslant\cdots E_b(\bk)\leqslant\cdots,
	\end{equation}
	and $E_b(\bk)\rightarrow+\infty$ as $b\rightarrow+\infty$. Moreover, for each $\bk\in\Omega^*$, $\{\phi_b(\bx;\bk)\}_{b\geqslant1}$ is a complete orthogonal set in $L^2(\br^2/\Lambda)$.
	
	\item The eigenvalues $E_b(\bk)$, referred as dispersion bands, are Lipschitz continuous via a similar discussion as the standard Schr\"odinger operator \cite{fefferman2014wave}.
	
	\item For each $b\geqslant1$, $E_b(\bk)$ sweeps out a closed real interval over $\bk\in\Omega^*$, the union of these intervals actually compose of the spectrum of $\hsgm$ in $L^2(\br^2)$,
		\begin{equation}
			spec(\hsgm)=\bigcup_{b\geqslant1,\bk\in\Omega^*}\Big[\min_{\bk\in\Omega^*}E_b(\bk),\max_{\bk\in\Omega^*}E_b(\bk)\Big].
		\end{equation}
	
\end{enumerate}

Actually, for each $\bk\in\Omega^*$, $\Phi_b(\by;\bk) $(or $\phi_b(\by;\bk)$) is smooth on $\Omega$. All Bloch eigenfunction set $\bigcup\limits_{b\geqslant1,\bk\in\Omega^*}\{\Phi_b(\by;\bk)\}$ forms a complete orthogonal set of $L^2(\br^2)$, i.e., for any $f\in L^2(\br^2)$, the following summation converges in $L^2-$norm,
\begin{align}
	f(\by)=\frac1{|\Omega^*|}\sum_{b\geqslant 1}\int_{\Omega^*}\widetilde{f}_b(\bk) \Phi_b(\by;\bk)d\bk,\quad \widetilde{f_b}(\bk)=\bla\Phi_b(\by;\bk), f(\by)\bra=\int_{\br^2} \overline{\Phi_b(\by;\bk)}f(\by)d\by,
\end{align}
and one can also establish a type of Plancherel theory
\begin{equation}
\|f(\by)\|^2_{L^2(\br^2)}=\frac1{|\Omega^*|}\sum_{b\geqslant 1}\int_{\Omega^*}\big| \widetilde{f}_b(\bk)\big|^2 d\bk.
\end{equation}
For any $s\geqslant 0$, if $f\in H^s(\br^2)$, it is natural to show the equivalence of $H^s-$norm,
\begin{equation}\label{id:hsnorm}
\|f(\by)\|^2_{H^s(\br^2)} \approx \|(I+\hsgm)^{\frac{s}{\sigma}}f(\by)\|^2_{L^2(\br^2)} = \frac1{|\Omega^*|}\sum_{b\geqslant 1}\int_{\Omega^*}\big| 1+E_b(\bk)\big|^{\frac{2s}{\sigma}}\big|\widetilde{f}_b(\bk)\big|^2 d\bk.
\end{equation}

Up to now, we have built the Floquet-Bloch theory for fractional Schr\"odinger operator. In the next, we will clarify the rotational invariance in the context of honeycomb potential.

\subsection{Honeycomb potential}

The rotational invariant property is a novel hypothesis in the honeycomb potential. We define the rotation operator $\mathcal{R}$ as follows. For any function $f(\by)$ defined on $\br^2$,
\begin{equation}
\mathcal{R}f(\by):=f(R^*\by),
\end{equation}
where $R$ is the $2\pi/3-$clockwise rotation matrix
\begin{equation}
	R=
\begin{pmatrix}
-\frac12 & \frac{\sqrt{3}}{2} \\
-\frac{\sqrt{3}}{2} & -\frac12
\end{pmatrix}.
\end{equation}

In present article we shall study the honeycomb potential in the sense of the following definition.
\begin{definition}[Honeycomb potential]\label{hc:potential}
	A real-valued function $V(\by)\in C^{\infty}(\br^2)$ is called as a honeycomb potential if the following properties hold:
	\begin{enumerate}
		\item[(1)] $V(\by)$ is even or inversion-symmetric, i.e., $V(-\by)=V(\by)$;
		\item[(2)] $V(\by)$ is $\Lambda-$periodic, i.e., $V(\by+\bv)=V(\by),~\forall~\bv\in\Lambda$;
		\item[(3)] $V(\by)$ is $\mathcal{R}-$invariant, i.e., $\mathcal{R}V(\by)= V(R^*\by)= V(\by)$.
	\end{enumerate}
\end{definition}

It will be seen later that $\mathcal{R}-$invariance plays a key role in the existence of the degenerate Dirac points. There are three high symmetry points with respect to $\mathcal{R}$ in $\Omega^*$. Indeed, 
\begin{equation}
\bk \in \Omega^*, ~\mathcal{R}\bk -\bk \in \Lambda^* \quad \text{if and only if } \quad \bk\in \big\{\boldsymbol{0},~ \frac13(\bk_1-\bk_2),~ -\frac13(\bk_1-\bk_2)\big\}.
\end{equation}
Denote $\gk=\frac13(\bk_1-\bk_2)$ and $\gk'=-\frac13(\bk_1-\bk_2)$. One can figure out later that the two points $\gk$ and $\gk'$ are essential. In this paper, we only consider $\gk$ and the analysis for $\gk'$ is similar. Thanks to the high symmetry, $\mathcal{R}: L^2_{\gk}(\br^2/\Lambda)\rightarrow L^2_{\gk}(\br^2/\Lambda)$ is isometric, and $\mathcal{R}^3=Id $, thus the eigenvalues of $\mathcal{R}$ are $1,~\tau=e^{\text{i}2\pi/3}$ and $\bar{\tau}$. We can divide $L^2_{\gk}(\br^2/\Lambda)$ into an orthogonal direct sum of eigenspaces of $\mathcal{R}$, see also in \cite{fefferman2012honeycomb,lee2019elliptic},
\begin{equation}\label{lk2decomp}
L^2_{\gk}(\br^2/\Lambda)=L^2_{\gk,1}\oplus L^2_{\gk,\tau}\oplus L^2_{\gk,\bar{\tau}},~
L^2_{\gk,\nu}=\{f\in L^2_{\gk}(\br^2/\Lambda)|\mathcal{R}f=\nu f\},~ \nu=1,~\tau,~\bar{\tau}.
\end{equation}

In addition, we have the following proposition.
\begin{proposition}
	The fractional Schr\"odinger operator $\hsgm$ commutes with $\mathcal{R}$ on $H^{\sigma}_{\gk}(\br^2/\Lambda)$. Namely, $[\mathcal{R},\hsgm]=\mathcal{R}\hsgm-\mathcal{R}\hsgm$ vanishes.
\end{proposition}

\emph{Proof.} For any $\Phi(\by)\in H^{\sigma}_{\gk}(\br^2/\Lambda)$, a direct calculation yields $\mathcal{R}\Phi(\by)\in H_{\gk}^{\sigma}(\br^2/\Lambda)$. Indeed, $\forall~\bv\in\Lambda$,
\begin{equation}
\mathcal{R}\Phi(\by+\bv) = \Phi(R^*(\by+\bv))= e^{\text{i}\gk\cdot R^*\bv}\Phi(R^*\by)	= e^{\text{i}R\gk\cdot \bv}\Phi(R^*\by)
= e^{\text{i}\gk\cdot \bv}\mathcal{R}\Phi(\by).
\end{equation}
Then, for any $\bm\in \mathbb{Z}^2$, the Fourier transform of $\mathcal{R}\Phi(\by)$ gives
\begin{align}\label{fourier:coeffrot}
\hat{\mathcal{R}\Phi}(\gk+\bm\vec{\bk}) = \int_{\Omega}\Phi(R^*\by)e^{-\text{i}R^*(\gk+\bm\vec{\bk})\cdot R^*\by}d\by = \int_{R^*\Omega}\Phi(\by)e^{-\text{i}R^*(\gk+\bm\vec{\bk})\cdot \by}d\by.
\end{align}
Using the fact of integral invariance between the domain $R^*\Omega$ and $\Omega$, we obtain $\hat{\mathcal{R}\Phi}(\gk+\bm\vec{\bk})=\hat{\Phi}\big(R^*(\gk+\bm\vec{\bk})\big)$.

Recalling that $V(\by)$ is $\mathcal{R}-$invariant, we immediately obtain
\begin{align}
\nonumber	\hsgm\big(\mathcal{R}\Phi(\by)\big) =& \sum_{\bm\in\mathbb{Z}^2}|\gk+\bm\vec{\bk}|^{\sigma}\hat{\Phi}\big(R^*(\gk+\bm\vec{\bk})\big)e^{\text{i}(\gk+\bm\vec{\bk})\cdot\by}+V(\by) \Phi(R^*\by) \\
\nonumber	=& \sum_{\bm\in\mathbb{Z}^2}|\gk+\bm\vec{\bk}|^{\sigma}\hat{\Phi}(\gk+\bm\vec{\bk})e^{\text{i}(\gk+\bm\vec{\bk})\cdot R^*\by}+V(\by) \Phi(R^*\by) \\
=&~ \mathcal{R}\big(\hsgm \Phi(\by)\big).
\end{align}

\section{Linear Spectrum---Existence of Dirac Points}\label{sec:dirac}

In this section, we give the existence theorem of conically degenerate points, also known as Dirac points, on the spectra of the fractional Schr\"odinger operator $\hsgm$ with a honeycomb potential acting on $L_{\gk}^2(\br^2/\Lambda)$.

\begin{theorem}[Dirac point]\label{thm:diracpoint}
	Let $\hsgm=(-\Delta)^{\frac{\sigma}{2}}+V(\by)$, where $V(\by)$ is a honeycomb potential followed by Definition \ref{hc:potential}. Assume 
	\begin{enumerate}
		\item\label{assum1} $\hsgm$ has a two-fold degenerate $L^2_{\gk}-$eigenvalue $E_D$, i.e., there exists $b_*\geqslant1$ such that $E_{b_*}(\gk)=E_{b_*+1}(\gk)=E_D$, and $E_b(\gk)\neq E_D$ when $b\neq b_*, b_*+1$.
		\item There exists a normalized eigenfunction $\Phi_1(\by)\in L^2_{\gk,\tau}$ corresponding to $E_D$.
		\item\label{assum3} The following non-degeneracy condition holds:
		\begin{equation}\label{velocity}
		v_F^{\sigma}=\frac12\overline{\bla\Phi_1(\by), \mathrm{i}\boldsymbol{p}^{\sigma}\Phi_2(\by)\bra}_{\Omega}\cdot
		\begin{pmatrix}
		1\\
		\mathrm{i}
		\end{pmatrix}>0,
		\end{equation}
		where $\Phi_2(\by)=\overline{\Phi_1(-\by)}\in L^2_{\gk,\bar{\tau}}$ and the operator $\boldsymbol{p}^{\sigma}$ is defined later in \eqref{psgm}.		
	\end{enumerate}
Then there exists a constant $q_0>0$ such that for any $|\bk-\gk|<q_0$, two distinct eigenvalue bands conically intersect at $(\gk, E_D)$ which is referred as to a Dirac point. Namely,
\begin{align}\label{q0thm}
	E_{b_*+1}(\bk)-E_D &= +v_F^{\sigma}|\bk-\gk|~\big(1+e_{b_*+1}(\bk-\gk)\big),\\
	E_{b_*}(\bk)-E_D &= -v_F^{\sigma}~|\bk-\gk|~\big(1+e_{b_*}(\bk-\gk)\big),
\end{align}
where $|e_{b_*,b_*+1}(\bk-\gk)|< C|\bk-\gk|$ as $|\bk-\gk|<q_0$.
\\

\end{theorem}

\begin{remark}\label{remark:realvalue}
The above definition on $v_F^{\sigma}$ depends on the specific choice of $\Phi_1(\by)$. However, the value  $\left|\frac12\overline{\bla\Phi_\tau(\by), \mathrm{i}\boldsymbol{p}^{\sigma}\Phi_{\bar{\tau}}(\by)\bra}_{\Omega}\cdot\binom{1}{\mathrm{i}}\right|$ is fixed and independent of the choice of two normalized eigenfunctions $\Phi_\nu(\by)\in L^2_{\gk,\nu},~\nu=\tau,\bar{\tau}$. Indeed, if $v_F^{\sigma}$ is complex-valued, we could revise $\widetilde{\Phi}_1(\by)=e^{-\frac12\mathrm{i}\arg v_F^{\sigma}}\Phi_1(\by)$ and  $\widetilde{v_F^{\sigma}}=-\frac12\overline{\bla\widetilde{\Phi}_1(\by), \mathrm{i}\boldsymbol{p}^{\sigma}\widetilde{\Phi}_2(\by)\bra}_{\Omega}\cdot\binom{1}{\mathrm{i}}=|v_F^{\sigma}|>0$. Henceforth, we shall assume this choice of eigenfunctions by dropping superscript tildes. We refer the readers to Remark 2 in \cite{lee2019elliptic} for details.
\end{remark}

\begin{remark}
Consider the fractional Schr\"odinger operator $(-\Delta)^{\frac{\sigma}{2}}+\epsilon V_0(\by)$ with $V_0(\by)$ be a honeycomb potential. The assumptions \eqref{assum1}-\eqref{assum3} are satisfied for the sufficiently small $\epsilon$. The brief proof is given in Appendix.  However,  for a generic $\epsilon$, the assumptions \eqref{assum1}-\eqref{assum3} are satisfied almost all $\epsilon\in \mathbb{R}$ except a countable discrete set. We refer readers to \cite{fefferman2017topologically,fefferman2012honeycomb} for detailed proofs. 	
\end{remark}

\begin{proof}
It is evident that $\Phi_2(\by)=\overline{\Phi_1(-\by)}$ is also an eigenfunction of $\hsgm$ associated with $E_D$. Note that $\Phi_2(\by)\in L^2_{\gk,\bar{\tau}}$, so the two-degenerate eigenvalue $E_D$ has two linearly independent eigenfunctions $\Phi_1(\by)$, $\Phi_2(\by)$. Thus $E_D$ is NOT an $L^2_{\gk,1}-$eigenvalue of $\hsgm$. We only need to work on the subspace $L^2_{\gk,\tau}\oplus L^2_{\gk,\bar{\tau}}$.

Let $\Phi(\by;\bk)=e^{\text{i}\bk\cdot\by}\phi(\by;\bk)$ such that $\Phi(\by;\bk)\in L^2_{\bk}(\br^2/\Lambda)$ and $\phi(\by;\bk)\in L^2(\br^2/\Lambda)$. Supposing $|\bkappa|>0$ is small enough, we seek a non-trivial solution $\big(E(\gk+\bkappa),~\phi(\by;\gk+\bkappa)\big)$ to the periodic eigenvalue problem defined as \eqref{id:eigenprob:periodic} by Lyapunov-Schmidt reduction:
\begin{align}\label{expan1}
\nonumber	\hsgm(\gk+\bkappa)\phi(\by;\gk+\bkappa)&= \sum_{\bm\in\mathbb{Z}^2}|\gk+\bm\vec{\bk}+\bkappa|^{\sigma}\hat{\phi}(\bm\vec{\bk};\gk+\bkappa)e^{\text{i}\bm\vec{\bk}\cdot\by}+V(\by)\phi(\by;\gk+\bkappa) \\
&= E(\gk+\bkappa)\phi(\by;\gk+\bkappa),\quad (\by\in\br^2),
\end{align}
and when $\bkappa=\boldsymbol{0}$, $\hsgm(\gk)\phi_j(\by;\gk)=E_D\phi_j(\by;\gk),~j=1,~2$. Let $E(\gk+\bkappa)=E_D+E^{(1)}(\bkappa)$ and $E^{(1)}(\bkappa)\sim \mathcal{O}(|\bkappa|)$ since $E(\gk+\bkappa)$ is Lipschitz continuous. We decompose the eigenfunction into $\phi(\by; \gk+\bkappa)=\phi^{(0)}(\by)+\phi^{(1)}(\by)$ such that $\phi^{(0)}\in\ker(\hsgm(\gk)-E_D \text{I})$, $\phi^{(1)}\perp\ker(\hsgm(\gk)-E_D \text{I})$, i.e.,
\begin{equation}\label{phi:orthogonal}
\phi^{(0)}(\by)=a_1\phi_1(\by)+a_2\phi_2(\by)\quad\text{and}\quad \phi^{(1)}(\by)\perp\phi_j(\by),\quad j=1,~2.
\end{equation}
Here two parameters $a_1$ and $a_2$ are to be determined.

Next we shall expand the $\sigma-$power coefficients in \eqref{expan1} as $|\bkappa|>0$ sufficiently small,
\begin{equation}\label{expan:k}
|\gk+\bm\vec{\bk}+\bkappa|^{\sigma} = |\gk+\bm\vec{\bk}|^{\sigma}-\text{i} \bkappa\cdot \text{i}\sigma\big|\gk+\bm\vec{\bk}\big|^{\sigma-2}(\gk+\bm\vec{\bk})+c_{\bm,\bkappa}^{\sigma}|\bkappa|^2.
\end{equation}
One can observe the lower bound that $\forall \bm\in\bz^2$, $|\gk+\bm\vec{\bk}|\geqslant |\gk|>0$ holds, the second order Taylor expansion coefficients $c_{\bm,\bkappa}^{\sigma}$ will be uniformly bounded for all $\bm\in\mathbb{Z}^2$, see Lemma A.1 in \cite{hong2017new} for details.

\begin{remark}
	Actually, $c_{\bm,\bkappa}^{\sigma}\rightarrow 0$ as $\bm\rightarrow\infty$ here. Furthermore, one can verify the third part is in an order of $\mathcal{O}(|\bkappa|^{\sigma})$ for all $\bkappa\in\br^2$ \cite{hong2017new}.
\end{remark}

Given $\phi\in H^{\sigma}(\br^2/\Lambda)$, we define two operators $\boldsymbol{p}^{\sigma}$ and $q_{\bkappa}^{\sigma}$ on $H^{\sigma}(\br^2/\Lambda)$ by summing products of Fourier coefficients $\hat{\phi}(\bm\vec{\bk})$ together with the second and third terms in \eqref{expan:k} respectively as below
\begin{eqnarray}
	&\boldsymbol{p}^{\sigma}\phi(\by) = \text{i}\sigma\sum\limits_{\bm\in\mathbb{Z}^2}(\gk+\bm\vec{\bk})\big|\gk+\bm\vec{\bk}\big|^{\sigma-2}\hat{\phi}(\bm\vec{\bk})e^{\text{i}\bm\vec{\bk}\cdot\by}=e^{-\mathrm{i}\gk\cdot\by}\boldsymbol{p}^{\sigma}\Phi(\by);& \label{psgm} \\
	&\text{and}\quad\quad q_{\bkappa}^{\sigma}\phi(\by) = \sum\limits_{\bm\in\bz^2}c_{\bm,\bkappa}^{\sigma}\hat{\phi}(\bm\vec{\bk})e^{\text{i}\bm\vec{\bk}\cdot\by}.& \label{expan:qk}
\end{eqnarray}
Recall from \eqref{expan1}-\eqref{expan:qk}, it will deduce that
\begin{eqnarray}
\big( \mathcal{H}^{\sigma}(\gk)-E_D \big)\phi^{(1)}(\by)= \big(\text{i}\bkappa\cdot\boldsymbol{p}^{\sigma}-|\bkappa|^{2} q_{\bkappa}^{\sigma}+E^{(1)}(\bkappa)\big)\big(\phi^{(0)}(\by)+\phi^{(1)}(\by)\big).
\end{eqnarray}
We define the orthogonal projection operator $P_{\parallel}$ on $span\{\phi_1(\by), \phi_2(\by)\}$, $P_{\perp}=I-P_{\parallel}$. Then, we could get the following orthogonal decompositions:
\begin{align}\label{perp2}
\nonumber	P_{\perp}:\quad \big( \mathcal{H}^{\sigma}(\gk)-E_D \big)\phi^{(1)} =&~ P_{\perp}\big(\text{i}\bkappa\cdot\boldsymbol{p}^{\sigma}-|\bkappa|^{2} q_{\bkappa}^{\sigma}+E^{(1)}(\bkappa)\big)\phi^{(1)}(\by) \\
& +P_{\perp}\big(\text{i}\bkappa\cdot\boldsymbol{p}^{\sigma}-|\bkappa|^{2} q_{\bkappa}^{\sigma}+E^{(1)}(\bkappa)\big)\phi^{(0)}(\by)
\end{align}
and
\begin{align}\label{parallel}
\nonumber	P_{\parallel}:\quad 0 ~=&~ P_{\parallel}\big(\text{i}\bkappa\cdot\boldsymbol{p}^{\sigma}-|\bkappa|^{2}q_{\bkappa}^{\sigma}+E^{(1)}(\bkappa)\big)\phi^{(1)}(\by) \\
& +P_{\parallel}\big(\text{i}\bkappa\cdot\boldsymbol{p}^{\sigma}-|\bkappa|^{2}q_{\bkappa}^{\sigma}+E^{(1)}(\bkappa)\big)\phi^{(0)}(\by).
\end{align}

It is easy to check that $R_{\gk}^{\sigma}:=\big( \mathcal{H}^{\sigma}(\gk)-E_D \mathrm{I} \big)^{-1}$ is a bounded operator from $P_{\perp}L^2(\br^2/\Lambda)$ to $P_{\perp}H^{\sigma}(\br^2/\Lambda)$ and thus the operator $\Xi\big(\bkappa,E^{(1)}(\bkappa)\big)=R_{\gk}^{\sigma}P_{\perp}\big(\text{i}\bkappa\cdot\boldsymbol{p}^{\sigma}-|\bkappa|^{2}q^{\sigma}+E^{(1)}(\bkappa)\big)$ is bounded on $P_{\perp}H^{\sigma}(\br^2/\Lambda)$. Note that $\bkappa\rightarrow \boldsymbol{0}$ and therefore $E^{(1)}(\bkappa)\rightarrow 0$, $\phi^{(1)}(\by)$ can be resolved as
\begin{align}\label{phi1:conical}
\nonumber	\phi^{(1)}(\by)&=\big[\mathrm{I}-\Xi\big(\bkappa,E^{(1)}(\bkappa)\big)\big]^{-1}\Xi\big(\bkappa,E^{(1)}(\bk)\big)\phi^{(0)}(\by) \\
&= a_1\hat{C}\big(\bkappa, E^{(1)}(\bkappa)\big)\phi_1(\by)+a_2\hat{C}\big(\bkappa, E^{(1)}(\bkappa)\big)\phi_2(\by).
\end{align}
Here $\hat{C}\big(\bkappa, E^{(1)}(\bkappa)\big)=\big[\mathrm{I}-\Xi\big(\bkappa,E^{(1)}(\bkappa)\big)\big]^{-1}\Xi\big(\bkappa,E^{(1)}(\bkappa)\big): P_{\perp}H^{\sigma}(\br^2/\Lambda)\rightarrow P_{\perp}H^{\sigma}(\br^2/\Lambda)$ and
\begin{equation}
\|\hat{C}\big(\bkappa, E^{(1)}(\bkappa)\big)\phi_j(\by)\|_{H^{\sigma}(\br^2/\Lambda)}\leqslant C\big(|\bkappa|+|E^{(1)}(\bkappa)|\big),\quad j=1,~2.
\end{equation}

Substituting \eqref{phi1:conical} into \eqref{parallel}, we have a homogeneous system
\begin{equation}\label{eq:a1a2}
\mathcal{M}\big(\bkappa,E^{(1)}(\bkappa)\big)\binom{a_1}{a_2}=0,
\end{equation}
where $\mathcal{M}\big(\bkappa,E^{(1)}(\bkappa)\big)$ is a $2\times 2$
matrix with each component as follows:
\begin{align}\label{matrix}
\nonumber	\mathcal{M}_{ij}=&~\bla\phi_i(\by),\big(\text{i}\bkappa\cdot\boldsymbol{p}^{\sigma}-|\bkappa|^{2} q_{\bkappa}^{\sigma}+E^{(1)}(\bkappa)\big)\big(\hat{C}\big(\bkappa, E^{(1)}(\bkappa)\big)+1\big)\phi_j(\by)\bra_{\Omega} \\
\nonumber	=&~E^{(1)}(\bkappa)\dlt_{ij}+\bla\phi_i(\by),\big(\text{i}\bkappa\cdot \boldsymbol{p}^{\sigma}-|\bkappa|^{2}q_{\bkappa}^{\sigma}\big)\phi_j(\by)\bra_{\Omega} \\
&~+\bla\phi_i(\by),\big(\text{i}\bkappa\cdot \boldsymbol{p}^{\sigma}-|\bkappa|^{2}q_{\bkappa}^{\sigma}\big)\hat{C}\big(\bkappa, E^{(1)}(\bkappa)\big)\phi_j(\by)\bra_{\Omega}.
\end{align}
To obtain a nontrivial solution, $\mathcal{M}\big(\bkappa,E^{(1)}(\bkappa)\big)$ should be irreversible, i.e., $\det\mathcal{M}\big(\bkappa,E^{(1)}(\bkappa)\big)=0$.

Associated with Fourier coefficients of $\mathcal{R}\Phi_j$ in \eqref{fourier:coeffrot}, one can also observe that $\mathcal{R}\boldsymbol{p}^{\sigma}\Phi_j,~j=1,~2$ behave as
\begin{align}
\nonumber	\mathcal{R}\boldsymbol{p}^{\sigma}\Phi_j(\by) &= \text{i}\sigma\sum_{\bm\in\bz^2}(\gk+\bm\vec{\bk})\big|\gk+\bm\vec{\bk}\big|^{\sigma-2}\hat{\Phi}_j(\gk+\bm\vec{\bk})e^{\text{i}(\gk+\bm\vec{\bk})\cdot R^*\by} \\
\nonumber	&= \text{i}\sigma\sum_{\bm\in\bz^2}(\gk+\bm\vec{\bk})\big|\gk+\bm\vec{\bk}\big|^{\sigma-2}\hat{\Phi}_j(\gk+\bm\vec{\bk})e^{\text{i}R(\gk+\bm\vec{\bk})\cdot \by} \\
\nonumber	&= \text{i}\sigma\sum_{\bm\in\bz^2}R^*(\gk+\bm\vec{\bk})\big|R^*(\gk+\bm\vec{\bk})\big|^{\sigma-2}\hat{\Phi}_j(R^*(\gk+\bm\vec{\bk}))e^{\text{i}(\gk+\bm\vec{\bk})\cdot \by} \\
&= R^*\boldsymbol{p}^{\sigma}\mathcal{R}\Phi_j(\by).
\end{align}
Note that $\bla \phi_i(\by), \text{i}\boldsymbol{p}^{\sigma}\phi_j(\by)\bra_{\Omega}=\bla \Phi_i(\by), \text{i}\boldsymbol{p}^{\sigma}\Phi_j(\by)\bra_{\Omega},~i,~j=1,~2$. By the fact that $\Phi_1\in L^2_{\gk,\tau}$ and $\Phi_2\in L^2_{\gk,\bar{\tau}}$, we have
\begin{align}
\nonumber	\bla \Phi_i(\by), \text{i}\boldsymbol{p}^{\sigma}\Phi_j(\by)\bra_{\Omega} =&~ \bla \mathcal{R}\Phi_i(\by), \text{i}\mathcal{R}\boldsymbol{p}^{\sigma}\Phi_j(\by)\bra_{\Omega} \\
\nonumber	=&~ \bla \mathcal{R}\Phi_i(\by), \text{i}R^*\boldsymbol{p}^{\sigma}\mathcal{R}\Phi_j(\by)\bra_{\Omega} \\
=&~ R^*\bar{\tau}_i\tau_j\bla \Phi_i(\by), \text{i}\boldsymbol{p}^{\sigma}\Phi_j(\by)\bra_{\Omega}.
\end{align}
Here $\tau_1=\tau=e^{\text{i}2\pi/3},~\tau_2=\bar{\tau}$.

If $i=j$, then $\bar{\tau}_i\tau_j=|\tau_j|^2=1$. Therefore,
\begin{equation}
R\bla \Phi_j(\by), \text{i}\boldsymbol{p}^{\sigma}\Phi_j(\by)\bra_{\Omega}=\bla \Phi_j(\by), \text{i}\boldsymbol{p}^{\sigma}\Phi_j(\by)\bra_{\Omega},\quad j=1~\text{or}~2.
\end{equation}
However, $1$ is not an eigenvalue of matrix $R$, and then it indicates $\bla \Phi_j(\by), \text{i}\boldsymbol{p}^{\sigma}\Phi_j(\by)\bra_{\Omega}=0$, $j=1,~2$.

If $i=1,~j=2$, it gives $\bar{\tau}_1\tau_2=\tau$. Thus,
\begin{equation}
R\bla \Phi_1(\by), \text{i}\boldsymbol{p}^{\sigma}\Phi_2(\by)\bra_{\Omega} = \tau\bla \Phi_1(\by), \text{i}\boldsymbol{p}^{\sigma}\Phi_2(\by)\bra_{\Omega},
\end{equation}
and $(1,~\text{i})^T$ is an eigenvector with respect to $\tau$. Then, $v_F^{\sigma}$ satisfies
\begin{equation}
	\bla \Phi_1(\by), \text{i}\boldsymbol{p}^{\sigma}\Phi_2(\by)\bra_{\Omega}= -v_F^{\sigma}
\begin{pmatrix}
	1\\ \text{i}
\end{pmatrix}.
\end{equation}
In view of Remark \ref{remark:realvalue}, we choose the appropriate $\Phi_1(\bx),~\Phi_2(\bx)$ to deduce
\begin{equation}\label{id:phi1pphi2}
	v_F^{\sigma}=-\frac12\overline{\bla \Phi_1(\by), \text{i}\boldsymbol{p}^{\sigma}\Phi_2(\by)\bra_{\Omega}}\cdot
\begin{pmatrix}
	1\\ \text{i}
\end{pmatrix}>0,\quad
\bla \Phi_2(\by), \text{i}\boldsymbol{p}^{\sigma}\Phi_1(\by)\bra_{\Omega}=-v_F^{\sigma}
\begin{pmatrix}
	1\\ -\text{i}
\end{pmatrix}.
\end{equation}

Suppose that $v_F^{\sigma}$ is nonzero.  A direct calculation of $\det \mathcal{M}=0$ yields that
\begin{equation}
\big(E^{(1)}(\bkappa)\big)^2=|v_F^{\sigma}|^2|\bkappa|^2+\mathcal{O}\Big[|\bkappa|^{3}+|\bkappa|^{2}E^{(1)}(\bkappa)+|\bkappa|\big(E^{(1)}(\bkappa)\big)^2\Big].
\end{equation}
Thus,
\begin{equation}
E^{(1)}(\bkappa)=\pm v_F^{\sigma} |\bkappa|\big(1+e(\bkappa)\big),\quad |e(\bkappa)|\lesssim |\bkappa|.
\end{equation}
As $\bkappa\rightarrow\boldsymbol{0}$, It follows that two branches evolve conically in the neighborhood of $(\gk, E_D)$,
\begin{equation}
	E_{b_*}(\gk+\bkappa)=E_D- v_F^{\sigma}|\bkappa|+\mathcal{O}(|\bkappa|^{2}),\quad E_{b_*+1}(\gk+\bkappa)=E_D+ v_F^{\sigma}|\bkappa|+\mathcal{O}(|\bkappa|^{2}).
\end{equation}
\end{proof}

For notation convenience, we denote $b_*=-,~b_*+1=+$ in latter discussion. Furthermore, we also give a nontrivial normalized solution $(a_1,a_2)^T$ to \eqref{eq:a1a2}. Indeed, there exists $q_0>0$ such that for $0<|\bk-\gk|< q_0$, 
\begin{equation}
\begin{pmatrix}
	{a_1}_{\pm}\\
	{a_2}_{\pm}
\end{pmatrix}
=
\begin{pmatrix}
	\frac{\sqrt{2}}{2}\frac{\kappa_1+\mathrm{i}\kappa_2}{|\bkappa|}+\mathcal{O}(|\bkappa|)\\
	\pm\frac{\sqrt{2}}{2}+\mathcal{O}(|\bkappa|)
\end{pmatrix}
\end{equation}
Therefore two eigenfunctions corresponding to lower and upper bands are of the form
\begin{equation}\label{eigenfunc:expan}
	 \Phi_{\pm}(\by;\gk+\bkappa)=\frac{\sqrt{2}}{2}e^{\text{i}\bkappa\cdot\by}\Big[\frac{\kappa_1+\mathrm{i}\kappa_2}{|\bkappa|}\Phi_1(\by)\pm\Phi_2(\by)\Big]+\mathcal{O}_{H_{\gk}^{\sigma}(\br^2/\Lambda)}(|\bkappa|).
\end{equation}

According to above arguments, we carried out a clear spectra description varying around $\gk$ associated with the honeycomb fractional Schr\"odinger operator. The analogue spectra studies for standard elliptic operators are also studied in \cite{fefferman2012honeycomb,keller2018spectral,lee2019elliptic}.

Since $\{E_b(\bk)\}_{b\geqslant1,\bk\in\Omega^*}$ is an ordered sequence and Lipschitz continuous in $\Omega^*$, we can deduce the following corollary:

\begin{corollary}\label{corollary:br}
		
Let $(\gk,E_D)$ be a Dirac point given in Theorem \ref{thm:diracpoint}. There exists $q_1>0$ small, $b_r\geqslant +$, such that $\exists~ C>0$ infers
\begin{equation}
	\left\{
	\begin{aligned}
		|E_b(\bk)-E_D|>C,& \quad b\in\{1,\cdots,b_r\}\setminus\{\pm\},~|\bk-\gk|\leqslant q_1, \\
		E_b(\bk)-E_D>C,& \quad \forall~b>b_r,~\bk\in\Omega^*
	\end{aligned}
	\right.
\end{equation}

\end{corollary}

The following inner products of eigenfunctions will be used later.

\begin{proposition}\label{prop:phicubic}
Given $\Phi_1\in L^2_{\gk,\tau}(\br^2/\Lambda)$ and $\Phi_2\in L^2_{\gk,\bar{\tau}}(\br^2/\Lambda)$, the two normalized orthogonal eigenfunctions of $E_D$ as given in Theorem \ref{thm:diracpoint}, and $W\in C^{\infty}_{\gk}(\br^2/\Lambda)$ is real-valued and odd. Then
\begin{align}
	\bla\Phi_i(\by), W(\by)\Phi_j(\by)\bra_\Omega &= \vartheta(\dlt_{i1}-\dlt_{j2}),\quad i,~j=1,~2. \label{id:vartheta} \\
	\bla\Phi_i(\by), \overline{\Phi_j(\by)}\Phi_k(\by)\Phi_l(\by)\bra_\Omega &= \frac12 \big(b_1 \dlt_{ij}+b_2(1-\dlt_{ij})\big)(\dlt_{ik}\dlt_{jl}+\dlt_{il}\dlt_{jk}); \label{id:b1b2}
\end{align}
Here $b_1,~b_2$ and $\vartheta$ are real constants.
\end{proposition}

The proof is quite analogous to that shown in \cite{arbunich2018rigorous, hu2019linear} by considering rotational symmetry and will not be reproduced here. The conclusions in above together with \eqref{id:phi1pphi2} will be vital facts in deriving Dirac equations in the forthcoming effective dynamics. To keep things simple, we denote that
\begin{equation}\label{id:muijkl}
	\mu_{ijkl}=\mu\bla\Phi_i(\by), \overline{\Phi_j(\by)}\Phi_k(\by)\Phi_l(\by)\bra_\Omega
\end{equation}

If $\vartheta\neq0$, the fractional Schr\"odinger operator with a small modulation is changed into $(-\Delta)^{\frac{\sigma}{2}}+V(\by)+\veps W(\by)$. Followed by \cite{lee2019elliptic}, Dirac points will vanish and a local gap appears between two dispersion bands. This is related to navel topological phenomena in quantum mechanics and material innovations \cite{bal2017topological,bal2019continuous,fefferman2017topologically}.

\begin{figure}[h]
	\centering
	\includegraphics[height=6cm]{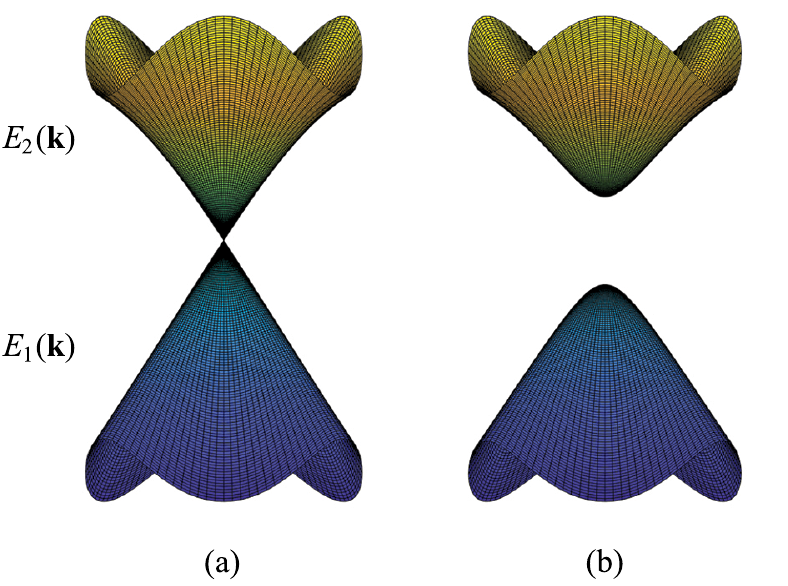}
	\caption{The two lowest dispersion band functions near $\gk$ for the fractional Schr\"odinger operator (a) $(-\Delta)^{\frac{\sigma}2}+V(\by)$ with $V(\by)$ giving in \eqref{numpoten}; (b) $(-\Delta)^{\frac{\sigma}2}+V(\by)+\veps W(\by)$ with a small perturbation $W(\by)$ giving in \eqref{nummodu}. Here $\veps=0.1$. Two dispersion bands touch each other at $\gk$ and form a local cone structure in (a). The cone disappears under perturbation in (b).}
	\label{fig1}
\end{figure}

As an example, we choose the following honeycomb potential 
\begin{equation}\label{numpoten}
	V(\by)=2\big[\cos(\bk_1\cdot\by)+\cos(\bk_2\cdot\by)+\cos((\bk_1+\bk_2)\cdot\by)\big].
\end{equation}
We numerically compute the quasi-periodic eigenvalue problem \eqref{id:eigenprob:pseodu} of  $\hsgm=(-\Delta)^{\frac{\sigma}{2}}+V(\by)$ with a Fourier collocation method \cite{guo2019bloch}. The first two bands $E_1(\bk)$ and $E_2(\bk)$ near $\gk$ are displayed in Figure \ref{fig1}(a). It shows that the two bands form a perfect cone in the small neighborhood of $\gk$ though the cone deforms due to high order effects in the far region.

We also numerically investigate the stability of Dirac point under an inversion-symmetry breaking perturbation. Let the perturbing potential $W(\by)$ be
\begin{equation}\label{nummodu}
	W(\by)=\big[\sin(\bk_1\cdot\by)+\sin(\bk_2\cdot\by)+\sin((-\bk_1-\bk_2)\cdot\by)\big].
\end{equation}
In Figure \ref{fig1}(b), we numerically compute the dispersion relations for the perturbed operator $(-\Delta)^{\frac{\sigma}2}+V(\by)+\veps W(\by)$ with $\veps=0.1$. It is evident that the two dispersion bands no longer intersect with each other and a local gap opens. This can be obtained by a similar perturbation argument as that in Proposition \ref{prop:phicubic} with the calculations \eqref{id:vartheta}. Since this is not the focus of this work, we omit the discussion. We remark that the papers \cite{berkolaiko2018symmetry,fefferman2017topologically,fefferman2012honeycomb,lee2019elliptic} contain detailed proofs for other operators.

\begin{figure}[h]
	\centering
	\includegraphics[height=5.3cm]{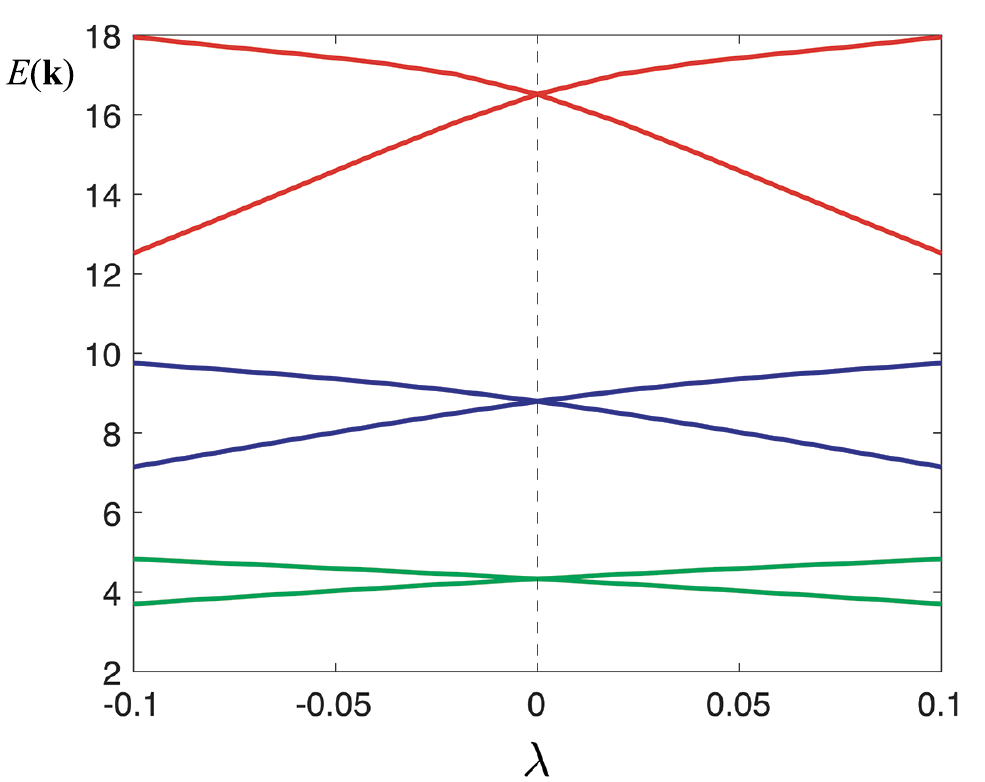}
	\caption{The dispersion curves $E_{1,2}(\gk+\lambda\bk_2)$ as $\lambda$ varies in $[-0.1,0.1]$ for different $\sigma$.   From the bottom up, the fractional exponents $\frac{\sigma}2=0.6$ (green), $0.8$ (blue), $1$ (red).}
	\label{fig2}
\end{figure}

We also examine the difference caused by the fractional exponent $\sigma$ numerically. The comparison is given in Figure \ref{fig2}. In the simulation, the honeycomb potential is given in \eqref{numpoten}. For the simplicity, we only display the dispersion curves along the $\bk_2$ direction. Figure \ref{fig2} shows the dispersion relation $E_{1,2}(\gk+\lambda\bk_2)$ as $\lambda\in[-0.1,0.1]$ with the fractional exponents $\frac{\sigma}2=0.6,~0.8,~1$ from the bottom to the top. We can see conical intersections in all three cases. However, the slope becomes steeper and steeper as $\sigma$ increases, and further the Dirac energy $E_D$ also increases. This can be explained for the shallow honeycomb potential discussed in Appendix \ref{appendix:intersec}, see Proposition \ref{thm:intersec}. It is an interesting problem for the generic case, but it is beyond the scope of the current work.

Before finishing this section, we emphasize that $\bp\Phi_{j}(\by),~j=1,~2$ are smooth by the following Corollary:
\begin{corollary}\label{coro:smooth}	
	If $\Psi\in C_{\gk}^{\infty}(\br^2/\Lambda)$, then
	$\bp\Psi\in C_{\gk}^{\infty}(\br^2/\Lambda)$.
\end{corollary}

\begin{proof}
	
Since $\Psi(\by)\in C_{\gk}^{\infty}(\br^2/\Lambda)$, for any integer $N\geqslant0$, it yields that
\begin{align}
(-\Delta)^N\Psi(\by)=\sum_{\bm\in\bz^2}|\gk+\bm\vec{\bk}|^N\hat{\Psi}(\gk+\bm\vec{\bk})e^{\text{i}(\gk+\bm\vec{\bk})\cdot\by}.
\end{align}
By Plancherel theorem, we can show that
\begin{equation}
\|(-\Delta)^N\Psi(\by)\|^2_{L^2(\Omega)}=\sum_{\bm\in\bz^2}|\gk+\bm\vec{\bk}|^{2N}|\hat{\Psi}(\gk+\bm\vec{\bk})|^2 <+\infty.
\end{equation}
Consequently, $\forall ~r\geqslant0,~\mathbf{m}\in\bz^2\setminus\{\boldsymbol{0}\}$, it gives $|\hat{\Psi}(\gk+\bm\vec{\bk})|\lesssim |\bm|^{-r}$.
Therefore $\forall~s>1$, one can show that
\begin{align}
	\|\bp\Psi(\by)\|^2_{H^s(\Omega)}\approx \sum_{\bm\in\bz^2}\big(1+|\gk+\bm\vec{\bk}|\big)^{2s}|\gk+\bm\vec{\bk}|^{\sigma-1}|\hat{\Psi}(\gk+\bm\vec{\bk})|^2 < +\infty.
\end{align}
According to the Moser-type Sobolev embedding theory, we complete the proof.

\end{proof}

\section{Nonlinear Effective Dynamics}\label{sec:main}

In this section, our main goal is to establish the effective dynamics of wave packet problem derived from fractional nonlinear Schr\"odinger equation with cubic nonlinearity \eqref{eq:fnls}. Suppose that $\bx=\veps \by\in\br^2$, and we denote the fractional Schr\"odinger operator $\hsgm_{\bx}=(-\veps^2\Delta_{\bx})^{\frac{\sigma}2}+V(\frac{\bx}{\veps})$ and $\hsgm_{\by}=(-\Delta_{\by})^{\frac{\sigma}2}+V(\by)$ for convenience.

The Bloch modes/eigenfunctions are highly oscillatory with respect to $\frac{\bx}{\veps}$, it is desirable to study the wave packet framework in the following weighted-Sobolev space $H_{\veps}^{s}$.
\begin{definition}\label{def:weightsob}
	Let $f(\bx)$ be a function defined on $\br^2$. For any $s\in\mathbb{N}$, $1\leqslant p<+\infty$, $0<\veps\leqslant1$, we say $f\in W_{\veps}^{s,p}$ if
	\begin{equation}
	\|f(\bx)\|_{W_{\veps}^{s,p}}=\Big[\sum_{|\bn|\leqslant s}\|\veps^{|\bn|}\partial_{\bx}^{\bn}f(\bx)\|_{L^p(\br^2)}^p\Big]^{\frac1p} < +\infty.
	\end{equation}
	Specifically, when $s=2$, suppose that $g(\by)=f(\veps\by)$. Then we have
	\begin{equation}\label{id:weightnorm}
	\|f(\bx)\|_{H_{\veps}^{s}}= \|\veps g(\by)\|_{H^{s}(\br^2)}.
	\end{equation}
\end{definition}

We first seek the leading order approximation to the wave packet problem described by fNLS \eqref{eq:fnls}. Assume the initial condition is spectrally concentrated at the Dirac point $(\gk,E_D)$.
We briefly show the well-posedness of envelopes---the solution to nonlinear Dirac equation with a varying mass \eqref{eq:diracalpha} in the following lemma:
\begin{lemma}\label{lemma:alpha}
	For any integer $s>1$, let $\alpha_{j_0}(\cdot)\in H^{s+1}(\br^2),~j=1,~2$ admit the initial condition to \eqref{eq:diracalpha}. Then, there exists $0<T_*<\infty$ such that the unique solution to \eqref{eq:diracalpha} satisfies
	\begin{align}\label{wellpose1}
	\alpha_{1,2}(t,\bx)&\in C\big([0,T_*), H^{s+1}(\br^2)\big)\cap C^1\big([0,T_*), H^{s}(\br^2)\big).
	\end{align}
	
	Suppose that $\alpha_{j_0}(\cdot)\in \mathcal{S}(\br^2),~j=1,~2$. Then, for any $t\in[0,T_*)$ and $0\leqslant N\leqslant s-1$, $N\in\mathbb{N}$,
	\begin{equation}\label{alpha:decay}
		\|\alpha_{1,2}(t,\cdot)\|_{W^{N,1}(\br^2)}<\infty.
	\end{equation}
\end{lemma}

This is a standard result by the hyperbolic system theory referred to \cite{kato1975cauchy, majda2012compressible, racke1992lectures}. The conclusion \eqref{alpha:decay} is a directly algebraic decay consequence of $\|(1+|\bx|^2)^{\frac{M}{2}}\ptl_{\bx}^{\bn}\alpha_j(t,\bx)\|_{L^{\infty}(\br^2)}<\infty$ for $M>2$, $|\bn|\leqslant N$, and it can be proved by the methods analogous to that used in \cite{xie2019wave}. Among the forthcoming derivation, we require $N\geqslant4$.

Now, the main theorem reads as
\begin{theorem}\label{thm:leading}
	Let $1<\sigma\leqslant2$, $0<\veps\ll 1$, $V(\cdot)\in C^{\infty}(\br^2/\Lambda)$ be a honeycomb potential, $(\gk, E_D)$ denote a Dirac point and $\Phi_{1,2}(\cdot)$ indicate the associated eigenfunctions. Suppose that the initial envelopes $\alpha_{j_0}(\cdot)\in \mathcal{S}(\br^2)$, $j=1,2$. For any integer $s>1$, $S=\max\{s+2,5\}$, $\alpha_{1,2}(t,\bx)\in C\big([0,T_*), H^{S+1}(\br^2)\big)\cap C^1\big([0,T_*), H^{S}(\br^2)\big)$. If the initial value to \eqref{eq:fnls} satisfies
	\begin{equation}
	\|\psi^{\veps}(0,\bx)-\alpha_{j_0}(\bx)\Phi_j(\frac{\bx}{\veps}) \|_{H_{\veps}^s}\leqslant C\veps,
	\end{equation}	
	Then, there exists $0<\tilde{T}<T_*$, the wave packet problem \eqref{eq:fnls} has a unique solution
	\begin{equation}
		\psi^{\veps}(t,\bx)\in C\big([0,\tilde{T}], H_{\veps}^s\big)
	\end{equation}and
	\begin{equation}
	\sup_{0\leqslant t\leqslant \tilde{T}}\|\psi^{\veps}(t,\bx)-e^{-\mathrm{i}E_D\frac{t}{\veps}}\alpha_j(t,\bx)\Phi_j(\frac{\bx}{\veps})\|_{H_{\veps}^s}\leqslant C\veps,
	\end{equation}
	where $C$ is a positive constant independent of $\veps$.
\end{theorem}

\begin{remark}
	This theorem gives the simplest form of asymptotic solution, which is parallel to that of linear wave packet problems \cite{fefferman2014wave,xie2019wave}. However, one main difference is the lifetime of validity may not reach effective dynamics as \eqref{wellpose1} due to the case of nonlinearity. To deal with this nonlinear effect, we proceed to derive a second order approximation. The proof of Theorem \ref{thm:leading} can be also followed by a modified spectral decomposition idea in the latter second order justification, which is referred from \cite{fefferman2014wave,xie2019wave}.
\end{remark}

Notice that the wave packet evolution is of two distinct spacial scales. Then we need clarify the product rule of fractional derivative for $\Gamma(\cdot)\in H^{\sigma}(\br^2)$ and $\Psi(\cdot)\in C_{\gk}^{\infty}(\br^2/\Lambda)$.

\begin{proposition}\label{prop:dividelaplacian}
	Let $\Gamma(\cdot)\in H^{\sigma}(\br^2)$ be algebraic decay at infinity, $\Psi(\cdot)\in C_{\gk}^{\infty}(\br^2/\Lambda)$ and $0<\veps\ll1$. The following product rule holds
	\begin{equation}\label{id:expan}
	 (-\veps^2\Delta_{\bx})^{\frac{\sigma}{2}}\big(\Gamma(\bx)\Psi(\frac{\bx}{\veps})\big)=\Gamma(\bx)(-\veps^2\Delta_{\bx})^{\frac{\sigma}{2}}\Psi(\frac{\bx}{\veps})-\veps\nabla_{\bx}\Gamma(\bx)\cdot \boldsymbol{p}^{\sigma}\Psi(\frac{\bx}{\veps})+q^{\sigma}[\Gamma(\bx)\Psi(\frac{\bx}{\veps})],
	\end{equation}
	where $\boldsymbol{p}^{\sigma}$ is defined in \eqref{psgm}, and
	\begin{align}	
	q^{\sigma}[\Gamma(\bx) \Psi(\frac{\bx}{\veps})]=&~ \sum_{\bn\in\mathbb{Z}^2}\int_{\br^2_{\bxi}} C_{\bn}(\veps\bxi) \widehat{\Gamma}(\bxi) e^{\mathrm{i}\bxi\cdot\bx}d\bxi \hat{\Psi}(\gk-\bn\vec{\bk})e^{\mathrm{i}(\gk-\bn\vec{\bk})\cdot\frac{\bx}{\veps}} .
	\end{align}	
	Here $C_{\bn}(\veps\bxi)=|\gk-\bn\vec{\bk}+\veps\bxi|^{\sigma}-|\gk-\bn\vec{\bk}|^{\sigma}+\mathrm{i}\sigma|\gk-\bn\vec{\bk}|^{\sigma-2}(\gk-\bn\vec{\bk})\cdot\mathrm{i}\veps\bxi$ and $|C_{\bn}(\veps\bxi)|\lesssim |\veps\bxi|^{\sigma}$ uniformly for all $\bn\in\bz^2$, $\bxi\in\br^2$.
	
	Moreover, for any $s\geqslant 0$, if $\Gamma(\cdot)\in H^{s+3}(\br^2)$, then
	\begin{align}\label{id:qsgm}
		\|q^{\sigma}[\Gamma(\bx) \Psi(\frac{\bx}{\veps})]\|_{H_{\veps}^s}=\veps^2 \mathcal{C}[\Gamma(\bx)\Psi(\frac{\bx}{\veps})]+\veps^3\mathcal{D}[\Gamma(\bx)\Psi(\frac{\bx}{\veps})].
	\end{align}
where $\|\mathcal{C}[\Gamma(\bx)\Psi(\frac{\bx}{\veps})]\|_{H_{\veps}^s}$ and $\|\mathcal{D}[\Gamma(\bx)\Psi(\frac{\bx}{\veps})]\|_{H_{\veps}^s}\sim \mathcal{O}(1)$.
	
\end{proposition}

\begin{proof}
By the fractional Laplacian defined in \eqref{id:delta},\eqref{laplacian:whole}, we introduce the construction by adopting of variable $\by$ for convenience. Then,
\begin{align}\label{frac:fourier}
\nonumber	&~(-\Delta)^{\frac{\sigma}{2}}\big[\Gamma(\veps\by)\Psi(\by)\big] \\
\nonumber	=&~ \frac{1}{4\pi^2}\int_{\br^2}|\bxi|^{\sigma}\int_{\br^2}\Gamma(\veps\gz)\Psi(\gz)e^{-\text{i}\bxi\cdot \gz}d\gz~ e^{\text{i}\bxi\cdot\by}d\bxi \\
=&~ \frac{1}{4\pi^2}\int_{\br^2}|\bxi|^{\sigma}\int_{\Omega}\sum_{\bm\in\mathbb{Z}^2}\Gamma\big(\veps(\gz+\bm\vec{\bv})\big)e^{\text{i}(\gk-\bxi)\cdot(\gz+\bm\vec{\bv})}\Psi(\gz)e^{-\text{i}\gk\cdot\gz}d\gz~ e^{\text{i}\bxi\cdot\by}d\bxi.
\end{align}
Utilizing Poisson-Summation formula, for any $\bn\in \bz^2$, we take Fourier transform to have
\begin{align}
\nonumber	 \frac{1}{\veps^2}\widehat{\Gamma}(\frac{\bxi-\gk+\bn\vec{\bk}}{\veps})&=\frac{1}{4\pi^2}\int_{\br^2}\Gamma(\veps\gz)e^{-\mathrm{i}(\bxi-\gk+\bn\vec{\bk})\cdot\gz}d\gz \\
&=\frac{1}{4\pi^2}\int_{\Omega}\sum_{\bm\in\mathbb{Z}^2}\Gamma\big(\veps(\gz+\bm\vec{\bv})\big)e^{\text{i}(\gk-\bxi)\cdot(\gz+\bm\vec{\bv})}e^{-\text{i}\bn\vec{\bk}\cdot\gz}d\gz.
\end{align}
By Parseval's identity, one can acquire that
\begin{equation}\label{possion:1}
\sum_{\bm\in\mathbb{Z}^2}\Gamma(\veps(\gz+\bm\vec{\bv}))e^{\text{i}(\gk-\bxi)\cdot(\gz+\bm\vec{\bv})}=\frac{4\pi^2}{|\Omega|}\sum_{\bn\in\mathbb{Z}^2}\frac{1}{\veps^2}\widehat{\Gamma}(\frac{\bxi-\gk+\bn\vec{\bk}}{\veps})e^{\mathrm{i}\bn\vec{\bk}\cdot\gz}.
\end{equation}

Substituting \eqref{possion:1} into \eqref{frac:fourier} yields
\begin{align}
\nonumber	(-\Delta)^{\frac{\sigma}{2}}\big[\Gamma(\veps\by)\Psi(\by)\big] =&~ \int_{\br^2}|\bxi|^{\sigma}\sum_{\bn\in\mathbb{Z}^2}\frac{1}{\veps^2}\widehat{\Gamma}(\frac{\bxi-\gk+\bn\vec{\bk}}{\veps})\frac{1}{|\Omega|}\int_{\Omega}\Psi(\gz)e^{-\text{i}(\gk-\bn\vec{\bk})\cdot\gz}d\gz~ e^{\text{i}\bxi\cdot\by}d\bxi \\
\nonumber	=&~ \int_{\br^2}|\bxi|^{\sigma}\sum_{\bn\in\mathbb{Z}^2}\widehat{\Gamma}(\frac{\bxi-\gk+\bn\vec{\bk}}{\veps}) e^{\text{i}\bxi\cdot\by}d\frac{\bxi}{\veps} ~ \hat{\Psi}(\gk-\bn\vec{\bk}) \\
=&~
\sum_{\bn\in\mathbb{Z}^2}\int_{\br^2}|\veps\tilde{\bxi}+\gk-\bn\vec{\bk}|^{\sigma}\widehat{\Gamma}(\tilde{\bxi}) e^{\text{i}\tilde{\bxi}\cdot\veps\by}d\tilde{\bxi}~\hat{\Psi}(\gk-\bn\vec{\bk})e^{\text{i}(\gk-\bn\vec{\bk})\cdot\by},
\end{align}
where $\veps\tilde{\bxi}=\bxi-\gk+\bn\vec{\bk}$. We drop the superscript of $\tilde{\bxi}$ and deduce that
\begin{equation}
|\veps\bxi+\gk-\bn\vec{\bk}|^{\sigma}=|\gk-\bn\vec{\bk}|^{\sigma}-\mathrm{i}\sigma|\gk-\bn\vec{\bk}|^{\sigma-2}(\gk-\bn\vec{\bk})\cdot\mathrm{i}\veps\bxi+C_{\bn}(\veps\bxi),
\end{equation}
where the residual $|C_{\bn}(\veps\bxi)|\leqslant C|\veps\bxi|^{\sigma}$ holds uniformly for any $1<\sigma\leqslant2$, $\bxi\in \br^2$ and $\bn\in\mathbb{Z}^2$, see the appendix in \cite{hong2017new} for details. Thus, we have
\begin{align}
\nonumber	&~(-\Delta)^{\frac{\sigma}{2}}\big[\Gamma(\veps\by)\Psi(\by)\big] \\
\nonumber	=&~\int_{\br^2}\widehat{\Gamma}(\bxi) e^{\text{i}\bxi\cdot\veps\by}d\bxi\sum_{\bn\in\mathbb{Z}^2}|\gk-\bn\vec{\bk}|^{\sigma}\hat{\Psi}(\gk-\bn\vec{\bk})e^{\text{i}(\gk-\bn\vec{\bk})\cdot\by}\\
\nonumber	&~ -\veps\int_{\br^2}\text{i}\bxi\widehat{\Gamma}(\bxi) e^{\text{i}\bxi\cdot\veps\by}d\bxi\cdot \text{i}\sigma\sum_{\bn\in\mathbb{Z}^2}|\gk-\bn\vec{\bk}|^{\sigma-2}(\gk-\bn\vec{\bk})\hat{\Psi}(\gk-\bn\vec{\bk})e^{\text{i}(\gk-\bn\vec{\bk})\cdot\by}\\
\nonumber	&~ +\sum_{\bn\in\mathbb{Z}^2}\int_{\br^2}C_{\bn}(\veps\bxi)\widehat{\Gamma}(\bxi) e^{\text{i}\bxi\cdot\veps\by}d\bxi~\hat{\Psi}(\gk-\bn\vec{\bk})e^{\text{i}(\gk-\bn\vec{\bk})\cdot\by} \\
	=&~ \Gamma(\veps\by)(-\Delta_{\by})^{\frac{\sigma}{2}}\Psi(\by)-\veps\nabla_{\bx}\Gamma(\veps\by)\cdot \boldsymbol{p}^{\sigma}\Psi(\by) + q^{\sigma}[\Gamma(\veps\by)\Psi(\by)],
\end{align}	
Namely, it leads to
\begin{equation}
	(-\veps^2\Delta_{\bx})^{\frac{\sigma}{2}}\big(\Gamma(\bx)\Psi(\frac{\bx}{\veps})\big)=\Gamma(\bx)(-\veps^2\Delta_{\bx})^{\frac{\sigma}{2}}\Psi(\frac{\bx}{\veps})-\veps\nabla_{\bx}\Gamma(\bx)\cdot \boldsymbol{p}^{\sigma}\Psi(\frac{\bx}{\veps})+q^{\sigma}[\Gamma(\bx)\Psi(\frac{\bx}{\veps})],
\end{equation}
where
\begin{align}
q^{\sigma}[\Gamma(\bx)\Psi(\frac{\bx}{\veps})]=\sum_{\bn\in\mathbb{Z}^2}\int_{\br^2} C_{\bn}(\veps\bxi) \widehat{\Gamma}(\bxi) e^{\text{i}\bxi\cdot\bx}d\bxi~ \hat{\Psi}(\gk-\bn\vec{\bk})e^{\text{i}(\gk-\bn\vec{\bk})\cdot\frac{\bx}{\veps}}.
\end{align}

Denote that $\bxi=(\xi_1,\xi_2)^T$ and $\gk-\bn\vec{\bk}=\big((\gk-\bn\vec{\bk})_1,(\gk-\bn\vec{\bk})_2\big)^T$. If $|\veps\bxi|\ll 1$, Taylor's formula yields that
\begin{align}
\nonumber	C_{\bn}(\veps\bxi)=&~\sigma(\sigma-2)|\gk-\bn\vec{\bk}|^{\sigma-4}(\gk-\bn\vec{\bk})_1^2(\veps\xi_1)^2 +\sigma|\gk-\bn\vec{\bk}|^{\sigma-2}(\veps\xi_1)^2 \\
\nonumber	&+\sigma(\sigma-2)|\gk-\bn\vec{\bk}|^{\sigma-4}(\gk-\bn\vec{\bk})_2^2(\veps\xi_2)^2 +\sigma|\gk-\bn\vec{\bk}|^{\sigma-2}(\veps\xi_2)^2 \\
\nonumber	&+\sigma(\sigma-2)|\gk-\bn\vec{\bk}|^{\sigma-4}(\gk-\bn\vec{\bk})_1(\gk-\bn\vec{\bk})_22\veps^2\xi_1\xi_2+\mathcal{O}(|\veps\bxi|^3) \\
	:=&~ \veps^2 \tilde{C}_{\bn}(\bxi)+ \mathcal{O}(|\veps\bxi|^3) .
\end{align}

Notice that $|\gk-\bn\vec{\bk}|$ has a lower positive bound $|\gk|$ for all $\bn\in\bz^2$. While $|\veps\bxi|<\frac12 |\gk|$, there exists a constant $0<C<+\infty$ such that
\begin{align}
|C_{\bn}(\veps\bxi)-\veps^2 \tilde{C}_{\bn}(\bxi)|< C|\veps\bxi|^3, \quad \forall~\bn\in\bz^2.
\end{align}

Thus, we can divide $q^{\sigma}[\Gamma(\bx)\Psi(\frac{\bx}{\veps})]$ into the following two parts,
\begin{align}
\nonumber	q^{\sigma}[\Gamma(\bx)\Psi(\frac{\bx}{\veps})]=&~\veps^2\sum_{\bn\in\mathbb{Z}^2}\int_{\br^2}  \tilde{C}_{\bn}(\bxi) \widehat{\Gamma}(\bxi) e^{\text{i}\bxi\cdot\bx}d\bxi\hat{\Psi}(\gk-\bn\vec{\bk})e^{\text{i}(\gk-\bn\vec{\bk})\cdot\frac{\bx}{\veps}} \\
\nonumber	&~+\sum_{\bn\in\mathbb{Z}^2}\int_{\br^2} \big(C_{\bn}(\veps\bxi)- \veps^2 \tilde{C}_{\bn}(\bxi)\big) \widehat{\Gamma}(\bxi) e^{\text{i}\bxi\cdot\bx}d\bxi\hat{\Psi}(\gk-\bn\vec{\bk})e^{\text{i}(\gk-\bn\vec{\bk})\cdot\frac{\bx}{\veps}} \\
:=&~ \veps^2 \mathcal{C}[\Gamma(\bx)\Psi(\frac{\bx}{\veps})]+\veps^3 \mathcal{D}[\Gamma(\bx)\Psi(\frac{\bx}{\veps})] .
\end{align}

More precisely, for any $s\geqslant0$, if $\Gamma(\cdot)\in H^{s+3}(\br^2)$ and $\Psi(\cdot)\in C_{\gk}^{\infty}(\br^2/\Lambda)$, we employ the fact that $\hat{\Psi}(\gk-\bn\vec{\bk})$ will be rapidly decay as $|\bn|\rightarrow+\infty$ to obtain
\begin{align}
\nonumber	\|\mathcal{C}[\Gamma(\bx)\Psi(\frac{\bx}{\veps})]\|_{H_{\veps}^s} &\leqslant C \sum_{\bn\in\bz^2} \Big\|\int_{\br^2}\tilde{C}_{\bn}(\bxi)\widehat{\Gamma}(\bxi) e^{\text{i}\bxi\cdot\bx}d\bxi\Big\|_{H^s(\br^2)} (1+|\gk-\bn\vec{\bk}|^s) |\hat{\Psi}(\gk-\bn\vec{\bk})| \\
	&\leqslant C<+\infty.
\end{align}
and 
\begin{align}
\nonumber 	&~\|\mathcal{D}[\Gamma(\bx)\Psi(\frac{\bx}{\veps})]\|_{H_{\veps}^s} \\
\nonumber	\leqslant&~ C\Big[\int_{|\bxi|\leqslant\frac{|\gk|}{2\veps}} (1+|\bxi|)^{2s}| \bxi|^6 |\widehat{\Gamma}(\bxi)|^2 d\bxi\Big]^{\frac12} + C\frac{1}{\veps^3}\Big[\int_{|\bxi|>\frac{|\gk|}{2\veps}} (1+|\bxi|)^{2s}|\veps \bxi|^{4} |\widehat{\Gamma}(\bxi)|^2 d\bxi\Big]^{\frac12} \\
\nonumber	\leqslant&~ C+ C\Big[\int_{|\bxi|>\frac{|\gk|}{2\veps}} (1+|\bxi|)^{2s}|\bxi|^6 |\widehat{\Gamma}(\bxi)|^2 d\bxi\Big]^{\frac12} \\
	\leqslant &~ C<+\infty.
\end{align}
This completes the proof of Proposition \ref{prop:dividelaplacian}.

\end{proof}

\subsection{Construction of the second order approximation}

According to the above proposition, we establish a more accurate approximate solution $\psi_{\veps}$ to handle the nonlinear effect in fNLS\eqref{eq:fnls}. The formal solution $\psi^{\veps}$ is in the form of
\begin{align}
	&\psi^{\veps}(t,\bx)=\psi_{\veps}(t,\bx)+\varphi(t,\bx), \label{id:asysolu11}\\
	\text{and}\quad&\psi_{\veps}(t,\bx)= e^{-\mathrm{i}E_D\frac{t}{\veps}}\Big[\alpha_{j}(t,\bx)\Phi_j(\frac{\bx}{\veps})+\veps\big(\beta_{j}(t,\bx)\Phi_j(\frac{\bx}{\veps})+ u_1^{\perp}(t,\bx)\big)\Big]. \label{id:asymsolu}
\end{align}
Here $\varphi(t,\bx)$ is the error, and $u_1^{\perp}(t,\bx)$ in the second term will contribute to eliminate the first order resonant residual.

Substituting \eqref{id:asysolu11}, \eqref{id:asymsolu} into fNLS \eqref{eq:fnls} implies
\begin{equation}\label{eq:varphi}
	\mathrm{i}\veps\ptl_t\varphi-\hsgm_{\bx}\varphi= e^{-\mathrm{i}E_D\frac{t}{\veps}}\big(R_1+R_2+R_3\big)	+N[\varphi].
\end{equation}
where $R_1$, $R_2$ are residuals with coefficients order in $\veps,~\veps^2$ respectively, all higher order residuals are involved in $R_3$, and $N[\varphi]$ is a polynomial of $\varphi$ and $\overline{\varphi}$.

Now that our goal is to show the second order approximation, $R_1$ determines the leading order term should be vanished.
Specifically,
\begin{equation}
	R_1(t,\bx)= \veps\Big[-\mathrm{i}\ptl_t\alpha_j\Phi_j-\nabla_{\bx}\alpha_j\cdot\bp\Phi_j +\kappa\alpha_jW\Phi_j+\mu\overline{\alpha_j}\alpha_k\alpha_l\overline{\Phi_j}\Phi_k\Phi_l\Big] +\veps(\hsgm_{\by}-E_D)u_1^{\perp}.
\end{equation}
Here and below, we adopt the Einstein notation for convenience with $j,~k,~l=1,~2$. Let $\ml=\hsgm_{\by}-E_D$ and $\mathcal{L}_{\perp}^{^{-1}}=P_{\perp}\ml^{^{-1}}P_{\perp}:L_{\gk}^2(\br^2/\Lambda)\rightarrow P_{\perp}L_{\gk}^2(\br^2/\Lambda)$. $u_1^{\perp}$ is defined as the orthogonal complement of $\ker(\ml)^{\perp}$ in the view of $L_{\gk}^2(\br^2/\Lambda)$,
\begin{align}\label{id:inver}
\nonumber	u_1^{\perp}(t, \bx) =&~ \nabla_{\bx}\alpha_{j}(t,\bx)\cdot \vbl \boldsymbol{p}^{\sigma}\Phi_j(\frac{\bx}{\veps}) -\kappa(\bx)\alpha_{j}(t,\bx)\vbl \big[W(\frac{\bx}{\veps})\Phi_j(\frac{\bx}{\veps})\big]\\
&-\mu\overline{\alpha_{j}(t,\bx)}\alpha_{k}(t,\bx)\alpha_{l}(t,\bx)\vbl\big[\overline{\Phi_j(\frac{\bx}{\veps})}\Phi_k(\frac{\bx}{\veps})\Phi_l(\frac{\bx}{\veps}) \big].
\end{align}
From the Corollary \ref{coro:smooth}, one can conclude that if $\Psi(\cdot)\in C_{\gk}^{\infty}(\br^2/\Lambda)$, then $\bp\Psi(\cdot)$, $\vbl\Psi(\cdot)\in C_{\gk}^{\infty}(\br^2/\Lambda)$ by Sobolev embedding, and therewith also quasi-periodic functions in \eqref{id:inver}. Notice that $\{\Phi_b(\cdot;\gk)\}_{b\geqslant1}$ performs a completely orthogonal basis of $L_{\gk}^2(\br^2/\Lambda)$. Then,
\begin{align}
\nonumber	(\hsgm_{\by}-E_D)u_1^{\perp}=&~ \nabla_{\bx}\alpha_{j}\cdot  \boldsymbol{p}^{\sigma}\Phi_j -\kappa\alpha_{j}W\Phi_j-\mu\overline{\alpha_{j}}\alpha_{k}\alpha_{l}\overline{\Phi_j}\Phi_k\Phi_l \\
\nonumber	& -\big[\nabla_{\bx}\alpha_{j}\cdot \bla\Phi_i, \boldsymbol{p}^{\sigma}\Phi_j\bra_{\Omega}-\kappa\alpha_{j}\bla\Phi_i, W\Phi_j\bra_{\Omega}-\mu\overline{\alpha_{j}}\alpha_{k}\alpha_{l}\bla\Phi_{i}, \overline{\Phi_j}\Phi_k\Phi_l\bra_{\Omega}\big]\Phi_i \\
	=&~\nabla_{\bx}\alpha_{j}\cdot  \boldsymbol{p}^{\sigma}\Phi_j -\kappa\alpha_{j}W\Phi_j-\mu\overline{\alpha_{j}}\alpha_{k}\alpha_{l}\overline{\Phi_j}\Phi_k\Phi_l+\mathrm{i}\ptl_t\alpha_{j}\Phi_j.
\end{align}
The last identity holds by the fact that the results stated in \eqref{velocity}, \eqref{id:vartheta}, \eqref{id:b1b2} and $\alpha_{1,2}(t, \bx)$ satisfy the nonlinear massive Dirac equation \eqref{eq:diracalpha}. Hence that $R_1=0$.

Now, $R_2$ turns to be the leading order term,
\begin{align}\label{id:R2}
\nonumber	R_2=&~ \veps^2 \Big[-\mathrm{i}\ptl_t\beta_j\Phi_j-\nabla_{\bx}\beta_j\cdot \boldsymbol{p}^{\sigma}\Phi_j+\kappa\beta_j W\Phi_j+\mu\big(\overline{\beta_j}\alpha_k\alpha_l+2\overline{\alpha_j}\alpha_k\beta_l\big)\overline{\Phi_j}\Phi_k\Phi_l +\mathcal{C}[\alpha_j\Phi_j]\\
&~ -\text{i}\ptl_t u_{1}^{\perp} -\nabla_{\bx}\cdot\bp(u_{1}^{\perp})+\kappa W u_{1}^{\perp} +\mu \big(\overline{u_{1}^{\perp}}\alpha_j\alpha_k\Phi_j\Phi_k+2\overline{\alpha_j}\alpha_k\overline{\Phi_j}\Phi_k u_{1}^{\perp}\big)\Big].
\end{align}

While $R_3$ contains all higher order entries,
\begin{align}\label{id:R3}
\nonumber	R_3 =&~ \veps^{3} \Big[\mu\Big(\overline{\alpha_j\Phi_j}(\beta_j\Phi_j+u_1^{\perp})^2+2|\beta_j\Phi_j+u_1^{\perp}|^2 \alpha_j\Phi_j\Big)+\mathcal{C}[\beta_j\Phi_j+u_1^{\perp}]+\mathcal{D}[\alpha_j\Phi_j]\Big]\\
&+ \veps^{4}\Big[\mu|\beta_j\Phi_j+u_1^{\perp}|^2(\beta_j\Phi_j+u_1^{\perp})+ \mathcal{D}[\beta_j\Phi_j+u_1^{\perp}]\Big],
\end{align}
and $N[\varphi]$ is constituted by linear, quadratic and cubic terms of $\varphi$ and $\overline{\varphi}$,
\begin{align}\label{id:N}
N[\varphi](t, \bx) =\veps\Big[\kappa W\varphi+\mu\big(\overline{\varphi}{\psi_{\veps}}^2+2|{\psi_{\veps}}|^2\varphi+\overline{{\psi_{\veps}}}\varphi^2+2|\varphi|^2{\psi_{\veps}}+|\varphi|^2\varphi\big)\Big].
\end{align}

Moreover, the envelopes $\beta_{1,2}(t,\bx)$ admit the following linear Dirac equation system:
\begin{equation}\label{eq:beta}
\left\{
\begin{aligned}
\ptl_t\beta_1+v_F^{\sigma}(\ptl_{x_1}+\text{i}\ptl_{x_2})\beta_2+\text{i}\vartheta\kappa(\bx)\beta_1+\text{i}\mu_{1jkl}(\overline{\beta_j}\alpha_k\alpha_l+2\overline{\alpha_j}\alpha_k\beta_l)&=\sum_{\bn\in\bz^2}f_{\bn 1}[\balpha]+\Theta_1, \\
\ptl_t\beta_2+v_F^{\sigma}(\ptl_{x_1}-\text{i}\ptl_{x_2})\beta_1-\text{i}\vartheta\kappa(\bx)\beta_2+\text{i}\mu_{2jkl}(\overline{\beta_j}\alpha_k\alpha_l+2\overline{\alpha_j}\alpha_k\beta_l)&=\sum_{\bn\in\bz^2}f_{\bn 2}[\balpha]+\Theta_2 ,
\end{aligned}
\right.
\end{equation}
where the source term $f_{\bn i}[\balpha]$ are functions of $(t,\bx)$ for all $\bn\in\bz^2$, $i=1,~2$,
\begin{align}
\nonumber	f_{\bn i}[\balpha](t, \bx)&=-\mathrm{i}\int_{\br^2_{\bxi}} \tilde{C}_{\bn}(\bxi) \widehat{\alpha_j}(t, \bxi)e^{\mathrm{i}\bxi\cdot\bx}d \bxi~ \bla\Phi_i(\by),\hat{\Phi}_j(\gk-\bn\vec{\bk})e^{\text{i}(\gk-\bn\vec{\bk})\cdot\by} \bra_{\Omega} \\
	&= -\mathrm{i}|\Omega|\overline{\hat{\Phi}_i(\gk-\bn\vec{\bk})}\hat{\Phi}_j(\gk-\bn\vec{\bk}) \int_{\br^2_{\bxi}} \tilde{C}_{\bn}(\bxi) \widehat{\alpha_j}(t, \bxi)e^{\mathrm{i}\bxi\cdot\bx}d \bxi.
\end{align}
And, $\Theta_{1,2}(t,\bx)$ are defined as an inner product in $L^2_{\gk}(\br^2/\Lambda)$,
\begin{equation}
	\Theta_j(t,\bx)=\bla\Phi_j,\ptl_t u_{1}^{\perp} +\mathrm{i}\nabla_{\bx}\cdot\bp(u_{1}^{\perp})-\mathrm{i}\kappa W u_{1}^{\perp} -\mathrm{i}\mu \big(\overline{u_{1}^{\perp}}\alpha_j\alpha_k\Phi_j\Phi_k +2\overline{\alpha_j}\alpha_k\overline{\Phi_j}\Phi_k u_{1}^{\perp}\big)\bra_{\Omega}.
\end{equation}
When $\sigma=2$, one can observe $\beta_{1,2}$ satisfy the same linear Dirac equation as \cite{arbunich2018rigorous}. Notice that $f_{\bn i}[\balpha],~i=1,2$ behave as the second order derivative of $\balpha$. Using the same argument as Lemma \ref{lemma:alpha}, we also conclude the well-posedness of $\beta_{1,2}$ in the following lemma:
\begin{lemma}\label{lemma:beta}

For any integer $s>1$, assume that $\alpha_{j_0}(\bx)\in H^{s+3}(\br^2)$, $\beta_{j_0}(\bx)\in H^{s+1}(\br^2)$, $j=1,~2$ are initial values to Dirac equations \eqref{eq:diracalpha}, \eqref{eq:beta}. Let $0<T_*<\infty$ define the lifetime of $\alpha_{1,2}(t,\bx)$ by Lemma \ref{lemma:alpha}. Then, the Dirac equation \eqref{eq:beta} has a unique solution
\begin{align}\label{wellpose2}
\beta_{1,2}(t,\bx)&\in C\big([0,T_*), H^{s+1}(\br^2)\big)\cap C^1\big([0,T_*), H^{s}(\br^2)\big).
\end{align}

\end{lemma}

Moreover, if $\alpha_{j_0}(\bx),~\beta_{j_0}(\bx)\in \mathcal{S}(\br^2),~j=1,~2$, it also indicates that $ \|\beta_{1,2}(t,\cdot)\|_{W^{N,1}(\br^2)}<\infty$ for $t\in[0,T_*)$, $0\leqslant N\leqslant s-1$, $N\in\mathbb{N}$.

Now, the initial condition to fNLS \eqref{eq:fnls} denotes as $\psi^{\veps}(0, \bx) ={\psi_{\veps}}_0(\bx)+\varphi_0(\bx)$ and ${\psi_{\veps}}_0(\bx)$ is given as:
\begin{equation}\label{id:solu0}
{\psi_{\veps}}_0(\bx)=\alpha_{j_0}(\bx)\Phi_j(\frac{\bx}{\veps})+\veps \big(\beta_{j_0}(\bx)\Phi_j(\frac{\bx}{\veps})+u_{10}^{\perp}(\bx)\big),
\end{equation}
where $u_{10}^{\perp}=\nabla_{\bx}\alpha_{j_0}\cdot \vbl \boldsymbol{p}^{\sigma}\Phi_j -\kappa\alpha_{j_0}\vbl \big[W\Phi_j\big]-\mu\overline{\alpha_{j_0}}\alpha_{k_0}\alpha_{l_0}\vbl\big[\overline{\Phi_j}\Phi_k\Phi_l \big]$.

We conclude the main result of second order approximation as follows:
\begin{theorem}\label{thm:main}
	Let $s>1$ be an integer, $1<\sigma\leqslant2$. $V(\cdot)\in C^{\infty}(\br^2/\Lambda)$ is a honeycomb potential. Assume that $\alpha_{j_0}(\cdot)$, $\beta_{j_0}\in \mathcal{S}(\br^2)$, $j=1,2$ are initial datum to Dirac equations \eqref{eq:diracalpha}, \eqref{eq:beta}. For $S=\max\{s+2,5\}$, $\alpha_{1,2}(t,\bx)\in C\big([0,T_{*}), H^{S+3}(\br^2)\big)\cap C^1\big([0,T_{*}), H^{S+2}(\br^2)\big)$, $\beta_{1,2}(t,\bx)\in C\big([0,T_{*}), H^{S+1}(\br^2)\big) \cap C^1\big([0,T_*), H^{S}(\br^2)\big)$ with $T_*>0$ finite. Suppose the initial condition satisfies
	\begin{equation}
		\|\psi^{\veps}(0,\bx)-\psi_{\veps0}(\bx)\|_{H_{\veps}^s}\leqslant C\veps^2.
	\end{equation}
	Then there exists $\veps_0>0$, for any $0<\veps<\veps_0$, the wave packet problem described by fNLS \eqref{eq:fnls} has a unique solution
	\begin{equation}
		\psi^{\veps}(t,\bx)\in C\big([0,T_*),~ H^s_{\veps}\big),
	\end{equation}
	and for any $0<T<T_{*}$, the approximated solution \eqref{id:asymsolu} satisfies the following estimate
	\begin{equation}
		\sup_{0\leqslant t\leqslant T}\|\psi^{\veps}(t, \bx)-\psi_{\veps}(t,\bx)\|_{H^s_{\veps}} \leqslant C\veps^2,
	\end{equation}
	where $C$ is a positive constant independent of $\veps$.
\end{theorem}

\begin{proof}

Using the Duhamel's principle, we rewrite the error evolution in \eqref{eq:varphi} as an integral equation, i.e.,
\begin{align}\label{id:varphi}
\nonumber	\varphi(t, \bx) &= e^{-\mathrm{i}\hsgm_{\bx}\frac{t}{\veps}} \varphi_0(\bx)  -\frac{\mathrm{i}}{\veps}\int_{0}^{t}e^{\mathrm{i}(\hsgm_{\bx}-E_D)\frac{s}{\veps}} e^{-\mathrm{i}\hsgm_{\bx}\frac{t}{\veps}}(R_2+R_3)(s, \bx)+e^{\mathrm{i}\hsgm_{\bx}\frac{s-t}{\veps}}N[\varphi](s, \bx)ds \\
	&:= G_0(t, \bx)+G_2(t, \bx)+G_3(t, \bx)+Q[\varphi](t, \bx).
\end{align}

We prove the above proposition by first investigating the estimate of $Q[\varphi](t, \bx)$. By the relationship stated in \eqref{id:weightnorm}, it is easy to show the following conclusion since the standard Sobolev space $H^s(\br^2)$ is a Banach algebra when $s>1$. Then,
\begin{align}
\nonumber	\|N[\varphi](t, \bx)\|_{H_{\veps}^s} \leqslant&~C \veps\|\varphi\|_{H_{\veps}^s}\big(\|\kappa W\|_{W_{\veps}^{s,\infty}}+\|\alpha_j\Phi_j\|^2_{W_{\veps}^{s,\infty}}+\|\beta_j\Phi_j+u_1^{\perp}\|^2_{H_{\veps}^{s}}\big) \\
\nonumber	&+ C\|\varphi\|_{H_{\veps}^s}^2\big(\|\alpha_j\Phi_j\|_{W_{\veps}^{s,\infty}}+\|\beta_j\Phi_j+u_1^{\perp}\|_{H_{\veps}^{s}}\big) +C\frac{\|\varphi\|_{H_{\veps}^s}^3}{\veps}\\
\leqslant&~C\veps\big(\|\varphi\|_{H_{\veps}^s}+\frac{\|\varphi\|_{H_{\veps}^s}^3}{\veps^2}\big).
\end{align}
The last inequality we employ the Cauchy-Schwartz.

Notice that $\hsgm_{\bx}$ is a self-adjoint operator. The linear fractional Schr\"odinger group $e^{\mathrm{i}\hsgm_{\bx}\frac{t}{\veps}}$ is unitary and commutes with $\hsgm_{\bx}$, i.e., for any $f\in H^s_{\veps}$, it gives that
\begin{align}
\|e^{\mathrm{i}\hsgm_{\bx}\frac{t}{\veps}}f(\bx)\|_{L^2(\br^2)}= \|f(\bx)\|_{L^2(\br^2)},
\end{align}
and
\begin{equation}\label{id:normequiv}
	\|e^{\mathrm{i}\frac{t}{\veps}\hsgm_{\bx}}f(\bx)\|_{H^s_{\veps}} \approx \|e^{\mathrm{i}\hsgm_{\bx}\frac{t}{\veps}}f(\bx)\|_{L^2(\br^2)}+\|(\hsgm_{\bx})^{\frac{s}{\sigma}} e^{\mathrm{i}\hsgm_{\bx}\frac{t}{\veps}}f(\bx)\|_{L^2(\br^2)} \approx \|f(\bx)\|_{H^s_{\veps}}.
\end{equation}

Thus, we can deduce the estimate of $Q[\varphi](t, \bx)$,
\begin{align}\label{id:qvarphi}
	\|Q[\varphi](t, \bx)\|_{H_{\veps}^s}=\|\frac{1}{\veps}\int_0^t e^{\mathrm{i}\hsgm_{\bx}\frac{s-t}{\veps}} N[\varphi](s, \bx)ds\|_{H_{\veps}^s} \leqslant C \int_0^t \|\varphi\|_{H_{\veps}^s}+\frac{1}{\veps^2}\|\varphi\|_{H_{\veps}^s}^3 ds
\end{align}

For $t\in[0,T]$, assume the following conclusion for $G_{0,2,3}(t,\bx)$ in \eqref{id:varphi} holds.
\begin{proposition}\label{prop:g2g3}
	Let $s>1$ and $0<T<T_*$, the following estimates of $G_{j}(t,\bx),~j=0,2,3$ defined in \eqref{id:varphi} holds	
	\begin{equation}
	\sup_{0\leqslant t\leqslant +\infty}\|G_0(t, \bx)\|_{H_{\veps}^s}\leqslant C\veps^{2},\quad
	\sup_{0\leqslant t\leqslant T}\|G_j(t, \bx)\|_{H_{\veps}^s}\leqslant C\veps^2,\quad j= 2,~3.
	\end{equation}

\end{proposition}
The detailed proof of above will be postponed in next section by a refined Bloch spectral decomposition arguments. Thus, we can conclude that for any $s>1$, $t\in[0,T]$,
\begin{align}\label{id:etanorm}
	\|\varphi(t, \bx)\|_{H_{\veps}^s} \leqslant C_1\veps^2 +C_2 \int_0^t \|\varphi\|_{H_{\veps}^s}+\frac{1}{\veps^2}\|\varphi\|_{H_{\veps}^s}^3 ds:=\veps \omega(t).
\end{align}

We next employ the nonlinear Gronwall's inequality \cite{yong1999singular} to derive the estimate of final result. Taking a derivative of $\omega(t)$ yields that
\begin{align}
\omega'(t)=C_2\big(\frac{\|\varphi(t,\bx)\|_{H_{\veps}^s}}{\veps}+\frac{\|\varphi(t, \bx)\|^3_{H_{\veps}^s}}{\veps^3}\big) \leqslant C_2\omega(t)+C_2 \omega^3(t),
\end{align}
and $\omega(0)=C_1\veps$. Then by multiplying $\exp(2C_2 t)\omega^{^{-3}}(t)$ on both sides and integrating from $0$ to $t$, it follows that
\begin{equation}\label{inequal:gronwall}
\exp(2C_2t)\omega^{^{-2}}(t) \geqslant \omega^{^{-2}}(0)-\big(\exp(2C_2t)-1\big).
\end{equation}

Thus, there exists $\veps_0>0$ sufficiently small, for all $t\in[0,T]$, $0<\veps<\veps_0$, one can deduce
\begin{align}
\omega^2(t)\leqslant 2 \exp(2C_2t)\omega^2(0)\leqslant 2C_1^2\exp(2C_2 T)\veps^2.
\end{align}

Consequently we acquire
\begin{align}
	\sup_{t\in[0,T]}\|\varphi(t, \bx)\|_{H_{\veps}^s}\leqslant C\veps^2.
\end{align}
\end{proof}

\begin{remark}
	The leading order term approximation theory in Theorem \ref{thm:leading} can be also derived in this strategy. In such setup, it is evident to see that $Q[\varphi]$ still evolves in the same way as \eqref{id:qvarphi}, and the estimate of explicit residuals as Proposition \ref{prop:g2g3} will reduce into order $\mathcal{O}(\veps)$. Consequently, the nonlinear Gronwall inequality is only applicable in a short lifetime in \eqref{inequal:gronwall}, since $w(0)$ will be of $\mathcal{O}(1)$.
\end{remark}

\section{Error Estimate of Explicit Parts}\label{sec:proof}

In this section, we will focus on the derivation in Proposition \ref{prop:g2g3}. The estimate to $G_0(t, \bx)$ is straightforward. We first establish the demonstration of $G_3(t, \bx)$.

\subsection{Estimate of $G_3$}

For any $t\in[0,T]$, recalling the higher order residual $R_3$ stated in \eqref{id:R3}, then,
\begin{align}
G_3(t, \bx)= -\frac{\mathrm{i}}{\veps}\int_0^t e^{\mathrm{i}(\hsgm-E_D)\frac{s}{\veps}}e^{-\mathrm{i}\hsgm_{\bx}\frac{t}{\veps}}R_3(s, \bx) ds
\end{align}

Notice that $\kappa(\bx)$ is smooth bounded on $\br^2$, $W(\cdot)\in C^{\infty}(\br^2/\Lambda)$, $\Phi_j(\cdot)\in C^{\infty}_{\gk}(\br^2/\Lambda)$ and  $\alpha_{1,2}(t,\cdot)\in C\big([0,T], H^{s+5}(\br^3)\big)$, $\beta_{1,2}(t,\cdot)\in C\big([0,T], H^{s+3}(\br^3)\big)$. Note the estimate \eqref{id:qsgm} given in Proposition \ref{prop:dividelaplacian}, for any $s>1$, $0\leqslant t\leqslant T$, we can deduce that
\begin{align}
\|R_3(t, \bx)\|_{H_{\veps}^s}\leqslant C\veps^{3},
\end{align}
and consequently we arrive at
\begin{align}
\sup_{0\leqslant t\leqslant T}\|G_3(t, \bx)\|_{H_{\veps}^s}\leqslant C\veps^{2}.
\end{align}

\subsection{Estimate of $G_2$}

Now, we turn to the key estimate of $G_2(t, \bx)$. Although, the coefficient order here is lower than that of $G_3(t, \bx)$, we shall modify the procedure of subtle Bloch spectral decomposition presented in \cite{fefferman2014wave,xie2019wave}, and improve the estimate of $G_2(t, \bx)$ up to order of exactly $\veps^2$ for $t\in[0,T]$.

In upcoming justifications, we will treat the error estimate by rescaling from macroscopic variable $\bx$ to microscopic variable $\by$. Recalling the equivalent relation of weighted/standard Sobolev space in \eqref{id:weightnorm} and let $\bx=\veps\by$, then we denote
\begin{equation}\label{id:rescaling1}
	\mathcal{G}(t,\by)=\veps G_2(t,\veps\by).
\end{equation}

By the Floquet-Bloch theory in Section \ref{sec:preli}, $\{\Phi_b(\by;\bk)\}_{b\geqslant 1,\bk\in\Omega^*}$ is complete in $L^2(\br^2)$. We claim the following expansion holds,
\begin{align}
	\mathcal{G}(t, \by)= \frac{1}{|\Omega^*|}\sum_{b\geqslant 1} \int_{\bk\in \Omega^*} \widetilde{\mathcal{G}}_b(t, \bk)\Phi_b(\by;\bk)d\bk,
\end{align}
where each component denotes as
\begin{align}
	\widetilde{\mathcal{G}}_b(t, \bk)=\bla \Phi_b(\by;\bk), \mathcal{G}(t, \by)\bra =-\mathrm{i}\int_0^t e^{\mathrm{i}(E_b(\bk)-E_D)\frac{s}{\veps}}e^{-\mathrm{i}E_b(\bk)\frac{t}{\veps}} \bla\Phi_b(\by;\bk), R_2(s, \veps\by)\bra ds,
\end{align}

Then, we separate $\mathcal{G}$ into two parts: $\mathcal{G}_\mathrm{D}$, the frequency components lie in the two spectral bands $(\bk, E_-(\bk))$ and $(\bk, E_+(\bk))$ conically intersecting at the Dirac point $(\gk, E_D)$, while $\mathcal{G}_\mathrm{D^c}$ indicates the frequency components lying in all the other spectral bands:
\begin{align}
\nonumber	\mathcal{G}(t, \by)=&~ \frac{1}{|\Omega^*|}\sum_{b\in\{\pm\}} \int_{\bk\in \Omega^*} \widetilde{\mathcal{G}}_b(t, \bk)\Phi_b(\by;\bk)d\bk+\frac{1}{|\Omega^*|}\sum_{b\notin\{\pm\}} \int_{\bk\in \Omega^*} \widetilde{\mathcal{G}}_b(t, \bk)\Phi_b(\by;\bk)d\bk  \\
:=&~ \mathcal{G}_\mathrm{D}(t, \by)+\mathcal{G}_{\mathrm{D^c}}(t, \by).
\end{align}
Furthermore, we need divide the conjugated area $\Omega^*$ of frequencies into ``near", ``middle" and ``far away from" $\gk$. To this end, we will use the indicator function $\chi(\cdot)$ defined below
\begin{equation}
\chi(\bk\in D):=\left\{
\begin{aligned}
&1,\quad\text{if}~\bk\in D, \\
&0,\quad\text{otherwise}.
\end{aligned}
\right.
\end{equation}
We are now going to decompose ${\mathcal{G}}_{\mathrm{D}}$ into 3 parts as follows:
\begin{align}
\nonumber	{\mathcal{G}}_\mathrm{D}(t, \by)=&~ \frac{1}{|\Omega^*|} \sum_{b\in\{\pm\}}\int_{\Omega^*}\Big[\chi(|\bk-\gk|<\veps)+\chi(\veps\leqslant |\bk-\gk|< q_0) \\
\nonumber	&\quad  +\chi(|\bk-\gk|\geqslant q_0)\Big]\widetilde{\mathcal{G}}_b(t, \bk)\Phi_b(\by;\bk) d\bk \\
:=&~{\mathcal{G}}_{\mathrm{D,I}}(t, \by)+{\mathcal{G}}_{\mathrm{D,II}}(t, \by)+{\mathcal{G}}_{\mathrm{D,III}}(t, \by).
\end{align}
While ${\mathcal{G}}_{\mathrm{D^c}}$ is composed of ${\mathcal{G}}_{\mathrm{D^c,I}}$ and ${\mathcal{G}}_{\mathrm{D^c,II}}$:
\begin{align}
\nonumber {\mathcal{G}}_{\mathrm{D^c,I}}(t, \by)=&~\frac{1}{|\Omega^*|}\sum_{b\in\{1,\cdots,b_r\}\setminus\{\pm\}}\int_{\Omega^*}\chi(|\bk-\gk|< q_1)\widetilde{\mathcal{G}}_b(t, \bk)\Phi_b(\by;\bk) d\bk  \\
	&~ +\frac{1}{|\Omega^*|}\sum_{b\geqslant b_r+1} \int_{\Omega^*}\widetilde{\mathcal{G}}_b(t, \bk)\Phi_b(\by;\bk) d\bk,\\
	{\mathcal{G}}_{\mathrm{D^c,II}}(t, \by)=&~\frac{1}{|\Omega^*|}\sum_{b\in\{1,\cdots,b_r\}\setminus\{\pm\}}\int_{\Omega^*}\chi(|\bk-\gk|\geqslant q_1)\widetilde{\mathcal{G}}_b(t, \bk)\Phi_b(\by;\bk) d\bk.
\end{align}
Here $q_0$ is defined in Section 3 to ensure \eqref{eigenfunc:expan} holds, and we choose $b_r\geqslant +$ and $q_1>0$ in Corollary \ref{corollary:br} such that $|E_b(\bk)-E_D|$ has a positive lower bound.

Due to the fact \eqref{id:hsnorm} and each $E_b(\bk)$ is Lipschitz continuous, for any $s\geqslant 0$, one can conclude
\begin{align}
\nonumber	\|\mathcal{G}(t, \by)\|^2_{H^s(\br^2)} &\approx \frac{1}{|\Omega^*|}\sum_{b\geqslant1}\int_{\Omega^*}(1+E_b(\bk))^{\frac{2s}{\sigma}}|\widetilde{\mathcal{G}}_b(t, \bk)|^2 d\bk \\
\nonumber	&\approx \|{\mathcal{G}}_{\mathrm{D}}(t, \by)\|^2_{L^2(\br^2)}+\|{\mathcal{G}}_{\mathrm{D^c}}(t, \by)\|^2_{H^s(\br^2)}\\
\nonumber	&\approx  \|{\mathcal{G}}_{\mathrm{D,I}}(t, \by)\|^2_{L^2(\br^2)}+\|{\mathcal{G}}_{\mathrm{D,II}}(t, \by)\|^2_{L^2(\br^2)}+\|{\mathcal{G}}_{\mathrm{D,III}}(t, \by)\|^2_{L^2(\br^2)} \\
&\quad +\|{\mathcal{G}}_{\mathrm{D^c,I}}(t, \by)\|^2_{H^s(\br^2)}+\|{\mathcal{G}}_{\mathrm{D^c,II}}(t, \by)\|^2_{L^2(\br^2)}.
\end{align}

In the following investigation, we will frequently employ Possion-Summation formula, and the proof has been given by \cite{fefferman2014wave}. We claim that for any $\Gamma(\cdot)$ algebraic decay at infinity and $\Psi(\cdot)\in C_{\gk}^{\infty}(\br^2/\Lambda)$, the inner product projects onto each eigenfunction could be decomposed into an infinite sequence summation on $L^2(\Omega)$, i.e.,
\begin{equation}
\bla \Phi_b(\by; \bk), \Gamma(\veps \by)\Psi(\by)\bra=\frac{1}{\veps^2}\frac{1}{|\Omega|}\sum_{\bm\in\bz^2}\widehat{\Gamma}(\frac{\bm\vec{\bk}+\bk-\gk}{\veps})\cdot \int_{\Omega}e^{\mathrm{i}(\bm\vec{\bk}+\bk-\gk)\cdot\by}\overline{\Phi_b(\by;\bk)}\Psi(\by)d\by,
\end{equation}
and if $\bxi\neq\boldsymbol{0}$, $\forall~N\geqslant 0$, we use integration by parts to get an upper bound
\begin{equation}\label{poisson:invbound}
|\widehat{\Gamma}(\bxi)|\leqslant C\frac{1}{|\bxi|^N}\|\Gamma(\cdot)\|_{W^{N,1}(\br^2)}<+\infty.
\end{equation}

Noticing the components in $R_2$, we formally denote $\mathbf{\Gamma}_{\ell}$ and $\mathbf{\Psi}_{\ell}$, $\ell=1,\cdots,6,\bn$ as follows:
\begin{eqnarray}
\nonumber		&\mathbf{\Gamma}_1=(-\mathrm{i}\ptl_t\beta_j,~-\nabla_{\bx}\beta_j,~\kappa\beta_j,~\mu(\overline{\beta_j}\alpha_k\alpha_l+2\overline{\alpha_j}\alpha_k\beta_l)),~ \mathbf{\Psi}_1=(\Phi_j,~\bp\Phi_j,~W\Phi_j,~\overline{\Phi_j}\Phi_k\Phi_l) &\\
\nonumber	&\mathbf{\Gamma}_2=-\mathrm{i}\ptl_t(\nabla_{\bx}\alpha_j,~ -\kappa\alpha_{j},~-\mu\overline{\alpha_{j}}\alpha_{k}\alpha_{l}),~ \mathbf{\Psi}_2=(\vbl\bp\Phi_j,~\vbl\big[W\Phi_j\big],~\vbl\big[\overline{\Phi_j}\Phi_k\Phi_l\big]) &\\
\nonumber	&\mathbf{\Gamma}_3=-\nabla_{\bx}(\nabla_{\bx}\alpha_j,~ -\kappa\alpha_{j},~-\mu\overline{\alpha_{j}}\alpha_{k}\alpha_{l}),~\mathbf{\Psi}_3=\bp(\vbl\bp\Phi_j,~\vbl\big[W\Phi_j\big],~\vbl\big[\overline{\Phi_j}\Phi_k\Phi_l\big])& \\
\nonumber	&\mathbf{\Gamma}_4=\kappa(\nabla_{\bx}\alpha_j,~ -\kappa\alpha_{j},~-\mu\overline{\alpha_{j}}\alpha_{k}\alpha_{l}),~\mathbf{\Psi}_4=W(\vbl\bp\Phi_j,~\vbl\big[W\Phi_j\big],~\vbl\big[\overline{\Phi_j}\Phi_k\Phi_l\big]) &\\
\nonumber	&\mathbf{\Gamma}_5=\mu\overline{(\nabla_{\bx}\alpha_j,~ -\kappa\alpha_{j},~-\mu\overline{\alpha_{j}}\alpha_{k}\alpha_{l})}\alpha_j\alpha_k,~\mathbf{\Psi}_5=\overline{(\vbl\bp\Phi_j,~\vbl\big[W\Phi_j\big],~\vbl\big[\overline{\Phi_j}\Phi_k\Phi_l\big])}\Phi_j\Phi_k &\\
\nonumber	&\mathbf{\Gamma}_6=2\overline{\alpha_j}\alpha_k(\nabla_{\bx}\alpha_j,~ -\kappa\alpha_{j},~-\mu\overline{\alpha_{j}}\alpha_{k}\alpha_{l}),~\mathbf{\Psi}_6=\overline{\Phi_j}\Phi_k(\vbl\bp\Phi_j,~\vbl\big[W\Phi_j\big],~\vbl\big[\overline{\Phi_j}\Phi_k\Phi_l\big])& \\
	&\mathbf{\Gamma}_{\bn}=\int_{\br^2_{\bzeta}}\tilde{C}_{\bn}(\bzeta) \widehat{\alpha_j}(t, \bzeta)e^{\mathrm{i}\bzeta\cdot\veps\by}d \bzeta,~\mathbf{\Psi}_{\bn}=\hat{\Phi}_j(\gk-\bn\vec{\bk})e^{\text{i}(\gk-\bn\vec{\bk})\cdot\by}.&
\end{eqnarray}
Moreover, thanks to the decay estimate of $\hat{\Phi}_{1,2}(\gk-\bn\vec{\bk})$ as $|\bn|\rightarrow +\infty$, it gives rise to a uniform summation convergence for all $\bm\in\bz^2$,
\begin{align}\label{id:poissonbound}
\nonumber	&\sum_{\bn\in\bz^2}\bav\widehat{\mathbf{\Gamma}}_{\bn}(t, \frac{\bm\vec{\bk}+\bk-\gk}{\veps})\int_{\Omega}e^{\mathrm{i}(\bm\vec{\bk}+\bk-\gk)\cdot\by}\overline{\Phi_b(\by;\bk)}\mathbf{\Psi}_{\bn}(\by)d\by\bav \\
\nonumber	\leqslant&~C\sum_{\bn\in\bz^2}\bav\tilde{C}_{\bn}(\frac{\bm\vec{\bk}+\bk-\gk}{\veps})\widehat{\alpha_j}(t, \frac{\bm\vec{\bk}+\bk-\gk}{\veps})\bav \bav\hat{\Phi}_j(\gk-\bn\vec{\bk})\bav \|e^{\mathrm{i}(\bm\vec{\bk}-\bn\vec{\bk})\cdot\by}\|_{L^2(\Omega)} \\
\leqslant&~C\bav\frac{\bm\vec{\bk}+\bk-\gk}{\veps}\bav^2\bav\widehat{\balpha}(t, \frac{\bm\vec{\bk}+\bk-\gk}{\veps})\bav.
\end{align}

We are now in a position to derive the estimate of ${\mathcal{G}}_{\mathrm{D}}(t, \by)$. For any $t\in [0,T]$, $b\in\{\pm\}$, we rewrite the Poisson-Summation sequence in view of $\bm=\boldsymbol{0}$ and $\bm\neq\boldsymbol{0}$.
\begin{align}
\nonumber	\bla\Phi_{\pm}(\by;\bk), R_2(t, \veps\by)\bra =&~\frac{1}{|\Omega|}\sum_{\bm\in\bz^2}\widehat{\mathbf{\Gamma}}_{\ell}(t, \frac{\bm\vec{\bk}+\bk-\gk}{\veps})\cdot\int_{\Omega}e^{\mathrm{i}(\bm\vec{\bk}+\bk-\gk)\cdot\by}\overline{\Phi_{\pm}(\by;\bk)}\mathbf{\Psi}_{\ell}(\by)d\by \\
=&~\frac{1}{|\Omega|}\Big(\sum_{\bm=\boldsymbol{0}}\cdot+\sum_{\bm\neq\boldsymbol{0}}\cdot\Big) :=\frac{1}{|\Omega|}\big(\textrm{I}_1(t,\bk)+\textrm{I}_2(t,\bk)\big).
\end{align}

If $\bm\neq \boldsymbol{0}$ and $|\bk-\gk|< q_0$, one can observe there exists a constant $C>0$ such that
\begin{equation}
|\bm\vec{\bk}+\bk-\gk|\geqslant C(1+|\bm|).
\end{equation}
We choose the integer $N>2$ in \eqref{poisson:invbound} to obtain
\begin{align}
\bav\chi(|\bk-\gk|< q_0)\textrm{I}_2\bav\leqslant C\veps^N.
\end{align}

Next, if $\bm=\boldsymbol{0}$, $|\bk-\gk|< q_0$, recalling \eqref{eigenfunc:expan} proposed in Section 3,  $\Phi_{\pm}(\by;\bk)$ satisfies the following expansion when $0<|\bk-\gk|< q_0$
\begin{equation}\label{id:eigenfexpan}
\Phi_{\pm}(\by;\bk)= \frac{e^{\mathrm{i}\bkappa\cdot\by}}{\sqrt{2}}\Big[\frac{\kappa_1+\mathrm{i}\kappa_2}{|\bkappa|}\Phi_1(\by)\pm\Phi_2(\by)+\mathcal{O}_{H_{\gk}^{\sigma}(\br^2/\Lambda)}(|\bkappa|)\Big],~\text{with}~\bkappa=\bk-\gk.
\end{equation}
Substituting above formula into $\textrm{I}_1$, it follows that
\begin{equation}
	\textrm{I}_1=\frac{\kappa_1+\mathrm{i}\kappa_2}{\sqrt{2}|\bkappa|}\widehat{\mathbf{\Gamma}}_{\ell}(t,\frac{\bkappa}{\veps})\cdot\bla\Phi_1(\by), \mathbf{\Psi}_{\ell}(\by)\bra_{\Omega} \pm \frac{1}{\sqrt{2}~}\widehat{\mathbf{\Gamma}}_{\ell}(t,\frac{\bkappa}{\veps})\cdot\bla\Phi_2(\by), \mathbf{\Psi}_{\ell}(\by)\bra_{\Omega}  +\mathfrak{R}(t,\frac{\bkappa}{\veps})\mathcal{O}(|\bkappa|),
\end{equation}
where $\mathfrak{R}(t,\frac{\bkappa}{\veps})\mathcal{O}(|\bkappa|)$ is the remainder.

According to the facts stated in \eqref{id:phi1pphi2}\eqref{id:vartheta}\eqref{id:b1b2}, all non-zero terms indicate that
\begin{align}
\nonumber	&\quad\widehat{\mathbf{\Gamma}}_{\ell}(t,\frac{\bkappa}{\veps})\cdot\bla\Phi_1(\by), \mathbf{\Psi}_{\ell}(\by)\bra_{\Omega}\\
\nonumber	=& -\mathrm{i}\widehat{\ptl_t\beta_1}(t,\frac{\bkappa}{\veps})-v_F^{\sigma}\big(\textrm{i}\widehat{\ptl_{x_1}\beta_2}(t,\frac{\bkappa}{\veps})-\widehat{\ptl_{x_2}\beta_2}(t,\frac{\bkappa}{\veps})\big)+\vartheta\widehat{\kappa\beta_1}(t,\frac{\bkappa}{\veps}) \\
	&+\mu_{1jkl}\big(2\widehat{\overline{\alpha_j}\alpha_k\beta_l}(t,\frac{\bkappa}{\veps})+\widehat{\overline{\beta_j}\alpha_k\alpha_l}(t,\frac{\bkappa}{\veps})\big)
+\mathrm{i}\sum_{\bn\in\bz^2}\widehat{f_{\bn}(\alpha_j)}(t,\frac{\bkappa}{\veps})g_{\bn 1}(\Phi_j)+\mathrm{i}\widehat{\Theta_1}(t,\frac{\bkappa}{\veps}),
\end{align}
and
\begin{align}
\nonumber	&\quad\widehat{\mathbf{\Gamma}}_{\ell}(t,\frac{\bkappa}{\veps})\cdot\bla\Phi_2(\by), \mathbf{\Psi}_{\ell}(\by)\bra_{\Omega} \\
\nonumber	=& -\mathrm{i}\widehat{\ptl_t\beta_2}(t,\frac{\bkappa}{\veps})-v_F^{\sigma}\big(\textrm{i}\widehat{\ptl_{x_1}\beta_2}(t,\frac{\bkappa}{\veps})+\widehat{\ptl_{x_2}\beta_2}(t,\frac{\bkappa}{\veps})\big)-\vartheta\widehat{\kappa\beta_2}(t,\frac{\bkappa}{\veps}) \\
	&+\mu_{2jkl}\big(2\widehat{\overline{\alpha_j}\alpha_k\beta_l}(t,\frac{\bkappa}{\veps})+\widehat{\overline{\beta_j}\alpha_k\alpha_l}(t,\frac{\bkappa}{\veps})\big)
+\mathrm{i}\sum_{\bn\in\bz^2}\widehat{f_{\bn}(\alpha_j)}(t,\frac{\bkappa}{\veps})g_{\bn 2}(\Phi_j)+\mathrm{i}\widehat{\Theta_2}(t,\frac{\bkappa}{\veps}) .
\end{align}
Since the Dirac equation \eqref{eq:beta} holds, we can see that both parts vanish on above. Then only the residual term $\mathfrak{R}(t,\frac{\bkappa}{\veps})\mathcal{O}(|\bkappa|)$ is left in $\textrm{I}_1$.

Notice that all $\widehat{\mathbf{\Gamma}}_{\ell}(t,\cdot)~(0\leqslant t\leqslant T)$ are bounded by Lemma \ref{lemma:alpha}, \ref{lemma:beta}. Then,
\begin{equation}
	\bav\chi(|\bk-\gk|<\veps)\textrm{I}_1(t,\bk)\bav=\bav\chi(|\bk-\gk|<\veps)\mathfrak{R}(t, \frac{\bk-\gk}{\veps})\mathcal{O}(|\bk-\gk|)\bav\leqslant C\veps.
\end{equation}
Thus, for any $t\in[0,T]$ we arrive at
\begin{align}
\nonumber	& \|\mathcal{G}_{\mathrm{D,I}}(t, \by)\|^2_{L^2(\br^2)} \\
\nonumber	\leqslant&~ C\sum_{b\in\{\pm\}} \int_{\Omega^*}\chi(0<|\bk-\gk|<\veps) \bav\int_0^t e^{\mathrm{i}(E_b(\bk)-E_D)\frac{s}{\veps}}e^{-\mathrm{i}E_b(\bk)\frac{t}{\veps}}\bla \Phi_b(\by;\bk), R_2(s, \veps\by)\bra ds\bav^2 d\bk \\
\nonumber	\leqslant&~ C\sum_{b\in\{\pm\}} \int_{\Omega^*}\chi(0<|\bk-\gk|<\veps)\bav\int_0^t \textrm{I}_1(s,\bk)+\textrm{I}_2(s,\bk)ds\bav^2 d\bk \\
\nonumber	\leqslant&~ C\sum_{b\in\{\pm\}} \int_{\Omega^*}\chi(0<|\bk-\gk|<\veps)\bav\int_0^t \veps+\veps^N ds\bav^2 d\bk \\
\leqslant&~ C \veps^4.
\end{align}

Whereas $\veps\leqslant|\bk-\gk|< q_0$, it leads to the following property of $\textrm{I}_1$ by recalling \eqref{poisson:invbound},
\begin{align}
	\Big|\chi(\veps\leqslant |\bk-\gk|< q_0)\textrm{I}_1(t,\bk)\Big|\leqslant \bav\mathfrak{R}(t, \frac{\bk-\gk}{\veps})\bav\mathcal{O}(|\bk-\gk|) \leqslant  C\frac{\veps^N~}{~|\bk-\gk|^{N-1}}.
\end{align}
Therefore, when $t\in[0,T]$ and $\veps\leqslant|\bk-\gk|< q_0$, one can obtain
\begin{align}
\nonumber	\|\mathcal{G}_{\mathrm{D,II}}(t, \by)\|^2_{L^2(\br^2)} \leqslant&~ C\int_{\Omega^*}\chi(\veps\leqslant |\bk-\gk|< q_0)\Big|\int_0^t \textrm{I}_1(s,\bk)+\textrm{I}_2(s,\bk)ds\Big|^2d\bk \\
\nonumber	\leqslant&~ C  \int_{\Omega^*}\chi(\veps\leqslant |\bk-\gk|< q_0) \frac{\veps^{2N}}{|\bk-\gk|^{2(N-1)}}d\bk \\
\nonumber	\leqslant&~ C \int_{\veps}^{q_0} \frac{\veps^{2N}}{r^{2(N-1)}}rdr = C \veps^{2N}\frac{1}{4-2N}r^{4-2N}\Big|_{\veps}^{q_0} \\
\leqslant&~ C\veps^{4}.
\end{align}
Here we choose $N>2$ to guarantee the boundedness.

Next we give the estimate of $\mathcal{G}_{\mathrm{D,III}}$. If $\bk\in\Omega^*$ and $|\bk-\gk|> q_0$, it is easy to show
\begin{equation}
|\bm\vec{\bk}+\bk-\gk|\geqslant C(1+|\bm|),\quad\forall \bm\in\bz^2.
\end{equation}
Then by invoking the Poisson-Summation formula, for any $t\in[0,T]$ we can deduce
\begin{align}
\chi(|\bk-\gk|> q_0)\big|\bla\Phi_{\pm}(\by;\bk), R_2(t, \veps\by)\bra\big|\leqslant C \sum_{\bm\in\bz^2}\frac{\veps^N}{(1+|\bm|)^N}.
\end{align}
Let $N>2$, we have the following bound as $t\in[0,T]$,
\begin{equation}
\|\mathcal{G}_{\mathrm{D,III}}(t, \by)\|^2_{L^2(\br^2)}\leqslant C \veps^4.
\end{equation}

$\newline$

Now, we turn to derive $\mathcal{G}_{\mathrm{D^c,I}}(t, \by)$, we first recall that
\begin{align}
\nonumber	\mathcal{G}_{\mathrm{D^c,I}}(t, \by)=&\sum_{b\in\{1,\cdots,b_r\}\setminus\{\pm\}}\int_{\Omega^*}\chi(|\bk-\gk|< q_1)\widetilde{\mathcal{G}}_b(t, \bk)\Phi_b(\by;\bk) d\bk  \\
&~~ +\sum_{b\geqslant b_r+1} \int_{\Omega^*}\widetilde{\mathcal{G}}_b(t, \bk)\Phi_b(\by;\bk) d\bk,
\end{align}
By employing the integration by parts, it yields
\begin{align}
\nonumber	\widetilde{\mathcal{G}}_b(t, \bk)=&-\mathrm{i}\int_0^t e^{\mathrm{i}(E_b(\bk)-E_D)\frac{s}{\veps}}e^{-\mathrm{i}E_b(\bk)\frac{t}{\veps}} \bla\Phi_b(\by;\bk), R_2(s, \veps\by)\bra ds \\
\nonumber	=&~\frac{-\veps}{E_b(\bk)-E_D}\Big[e^{-\mathrm{i}E_D\frac{t}{\veps}}\bla\Phi_b(\by;\bk), R_2(t, \veps\by)\bra-e^{-\mathrm{i}E_b(\bk)\frac{t}{\veps}}\bla\Phi_b(\by;\bk), R_2(0, \veps\by)\bra\Big] \\
&~+\frac{\veps}{E_b(\bk)-E_D}\int_0^t e^{\mathrm{i}(E_b(\bk)-E_D)\frac{s}{\veps}}e^{-\mathrm{i}E_b(\bk)\frac{t}{\veps}} \bla\Phi_b(\by;\bk), \ptl_s R_2(s, \veps\by)\bra ds.
\end{align}

According to Corollary \ref{corollary:br}, if $b\in\{1,\cdots,b_r\}\setminus\{\pm\},~|\bk-\gk|<q_1$ or $b\geqslant b_r+1,~\bk\in\Omega^*$, $|E_b(\bk)-E_D|$ have a positive lower bound. Thus, the following assertion holds for $t\in[0,T]$,
\begin{align}
\nonumber	&~\|\mathcal{G}_{\mathrm{D^c,I}}(t, \by)\|^2_{H^s(\br^2)}\\
\nonumber	\approx& \sum_{b\in\{1,\cdots,b_r\}\setminus\{\pm\}}\int_{\Omega^*}\chi(|\bk-\gk|< q_1) |\widetilde{\mathcal{G}}_b(t, \bk)|^2 d\bk+\sum_{b\geqslant b_r+1}\int_{\Omega^*}(1+E_b(\bk))^{\frac{2s}{\sigma}} |\widetilde{\mathcal{G}}_b(t, \bk)|^2 d\bk \\
\nonumber	\leqslant &~ C\veps^2(\|R_2(0,\veps \by)\|_{H^{s}(\br^2)}^2+\|R_2(t,\veps \by)\|_{H^{s}(\br^2)}^2) \\
\nonumber	&~ +C\veps^2 t \int_0^t \sum_{b\notin\{\pm\}}\int_{\Omega^*}(1+E_b(\bk))^{\frac{2s}{\sigma}}\bav\bla\Phi_b(\by;\bk),\ptl_s R_2(s, \veps \by)\bra\bav^2 d\bk ~ds \\
\nonumber	\leqslant&~ C\veps^4 + C\veps^2 t^2 \sup_{0\leqslant t\leqslant T}\|\ptl_t R_2(t,\veps \by)\|_{H^{s}(\br^2)}^2 \\
\leqslant&~ C\veps^4.
\end{align}

Now, it remains to show the estimate of $\mathcal{G}_{\mathrm{D^c,II}}(t, \by)$. When  $b\in\{1,\cdots,b_r\}\setminus\{\pm\}$, the spectral band $E_b(\bk)$ are uniform bounded, it is sufficient to justify under the $L^2-$norm.

Owing to $\bk\in\Omega^*$, $|\bk-\gk|\geqslant q_1$, for any $\bm\in\bz^2$,
\begin{equation}
	|\bm\vec{\bk}+\bk-\gk|\geqslant C(1+|\bm|).
\end{equation}
Therefore, an analogous argument used in $\mathcal{G}_{\mathrm{D,III}}$ indicates that
\begin{equation}
\sup_{0\leqslant t\leqslant T}\|\mathcal{G}_{\mathrm{D^c,II}}(t, \by)\|_{L^2(\br^2)}\leqslant C\veps^2.
\end{equation}

Finally, by invoking the fact of \eqref{id:weightnorm}\eqref{id:rescaling1}, we arrive at
\begin{equation}
	\sup_{0\leqslant t\leqslant T}\|G_2(t,\bx)\|_{H_{\veps}^s}=\sup_{0\leqslant t\leqslant T}\|\mathcal{G}(t,\by)\|_{H^s(\br^2)}\leqslant C\veps^2.
\end{equation}

Consequently, we finish the proof of Proposition \ref{prop:g2g3}, and therewith also Theorem \ref{thm:main}.

\hfill$\square$

~\newline

\noindent{\bf Acknowledgments:}
The authors thank Prof. Wen-An Yong and Dr. Pipi Hu for stimulating discussions. This work was supported by National Natural Science Foundation of China under grant $11871299$.

\begin{appendices}

\section{Dirac points in shallow honeycomb potentials}\label{appendix:intersec}

In this Appendix we prove that the assumptions of Theorem \ref{thm:diracpoint} can be satisfied for sufficient small honeycomb potentials. We treat $\hsgm_{\epsilon}=(-\Delta)^{\frac{\sigma}{2}}+\epsilon V(\by)$ as a perturbation to the fractional Laplacian $(-\Delta)^{\frac{\sigma}{2}}$ in $H_{\gk}^{\sigma}(\br^2/\Lambda)$. The following Lemma can be obtained by a direct justification.

\begin{lemma}
	
Let $1<\sigma\leqslant 2$, consider the $L_{\gk}^2(\br^2/\Lambda)-$eigenvalue problem of $(-\Delta)^{\frac{\sigma}{2}}$. The lowest eigenvalue $E^{(0)}=|\gk|^{\sigma}$ is of multiplicity three with the corresponding eigenspace spanned by  $e^{\mathrm{i}\gk\cdot\by}$, $e^{\mathrm{i}R\gk\cdot\by}$ and $e^{\mathrm{i}R^2\gk\cdot\by}$.

According to the orthogonal decomposition of $L_{\gk}^2(\br^2/\Lambda)$ under the rotational operator $\mathcal{R}$ in \eqref{lk2decomp}, and $\tau=e^{\mathrm{i}2\pi/3}$. Then, $E^{(0)}$ is a simple eigenvalue in $L_{\gk,\nu}^2$ with the corresponding normalized orthogonal eigenfunction as follows
\begin{equation}\label{phi0}
	\Phi^{(0)}_{\nu}(\by)=\frac{1}{\sqrt{3|\Omega|}~}\big[e^{\mathrm{i}\gk\cdot\by}+\bar{\nu}e^{\mathrm{i}R\gk\cdot\by}+\nu e^{\mathrm{i}R^2\gk\cdot\by}\big]\in L^2_{\gk,\nu}(\br^2/\Lambda),\quad \nu\in\{1,~\tau,~\bar{\tau}\}.
\end{equation}

\end{lemma}

Now we show the following spectral band degeneracy of $\gk-$quasi-periodic eigenvalue problem with a small amplitude potential, hence that conical existence of Dirac points. For a generic potential $V(\by)$, we refer the readers to \cite{fefferman2017topologically,fefferman2012honeycomb}

\begin{proposition}\label{thm:intersec}

Let $V(\by)$ be a honeycomb potential and $\hsgm_{\epsilon}=(-\Delta)^{\frac{\sigma}{2}}+\epsilon V(\by)$. Suppose that $\hat{V}(0,1)$, a Fourier coefficient of $V(\by)$, is not vanishing, i.e.,
\begin{equation}
	\hat{V}(0,1)=\frac{1}{|\Omega|}\int_{\Omega}e^{-\mathrm{i}\bk_2\cdot\by}V(\by)d\by\neq 0.
\end{equation}
Then there exists a constant $\epsilon_0>0$, and mappings $\epsilon\mapsto E_D^{\epsilon}$, $\epsilon\mapsto\Phi_1^{\epsilon}(\by)\in L^2_{\gk, \tau}$ such that for all $\epsilon\in(-\epsilon_0, \epsilon_0)$,
\begin{enumerate}
	\item[(1)] $E_D^{\epsilon}$ is a two-fold degenerate $L^2_{\gk}$-eigenvalue of $\hsgm_{\epsilon}$ with the following
	\begin{equation}
		E_D^{\epsilon}=E^{(0)}+\epsilon\big(\hat{V}(0,0)-\hat{V}(0,1)\big)+\mathcal{O}(\epsilon^2).
	\end{equation}
\item[(2)] $E_D^{\epsilon}$ is a simple $L^2_{\gk,\tau}-$eigenvalue of $\hsgm_{\epsilon}$ with the corresponding eigenspace spanned by the normalized eigenfunction $\Phi_1^{\epsilon}(\by)\in L^2_{\gk,\tau}$.
\item[(3)] The Dirac velocity $v^{\sigma}_{F,\epsilon}$ is approximated by
	\begin{equation}\label{id:velocity1}
		v^{\sigma}_{F,\epsilon}=\frac12\sigma(\frac{4\pi}{3})^{\sigma-1}+\mathcal{O}(\epsilon)\neq 0.
	\end{equation}
\end{enumerate}
Thus, by Theorem \ref{thm:diracpoint}, $(E_D^\epsilon, \gk)$ is a Dirac point.
\end{proposition}

\begin{proof}
	
Note that $E^{(0)}$ is a simple $L^2_{\gk,\nu}$-eigenvalue of the fractional Laplacian $(-\Delta)^{\frac{\sigma}{2}}$ for each $\nu\in \{1,~\tau,~\bar{\tau}\}$. Denote $E^{\epsilon}_\nu$ as the $L^2_{\gk,\nu}$-eigenvalue of $\hsgm_{\epsilon}$. Namely, consider the perturbed $L^2_{\gk,\nu}$-eigenvalue problem,
\begin{align}
\big((-\Delta)^{\frac{\sigma}{2}}+\epsilon V(\by)\big)\Phi^{(\epsilon)}_{\nu}(\by)= E_{\nu}^{\epsilon}\Phi^{(\epsilon)}_{\nu}(\by), \quad \Phi^{(\epsilon)}_{\nu}(\by)\in L^2_{\gk,\nu}.
\end{align}
Evidently, the eigenvalue $E^{\epsilon}_\nu$ remains simple in each subspace $L^2_{\gk,\nu}$ by a perturbation argument \cite{reed1978methods}. By symmetry, we know that $E^{\epsilon}_{\tau}=E^{\epsilon}_{\bar{\tau}}$ which is denoted by $E_D^{\epsilon}$. So we only need to show that  $E^{\epsilon}_{1}$ differs from $E^{\epsilon}_D$ and the Dirac velocity does not vanish.

With a Lyapunov-Schmidt reduction, we obtain that for a sufficiently small $\epsilon$,
\begin{equation}\label{eigen:epsilon}
	E^{\epsilon}_\nu =E^{(0)}+\epsilon \bla\Phi^{(0)}_{\nu}, V(\by)\Phi^{(0)}_{\nu}\bra_{\Omega}+O(\epsilon^2),~ \nu=1,~\tau,~\bar{\tau}.
\end{equation}
and
\begin{align}\label{id:velocity}
	v_{F}^{\sigma,\epsilon}=-\frac12\overline{\bla \Phi_1^{\epsilon}(\by), \mathrm{i}\boldsymbol{p}^{\sigma}\Phi_2^{\epsilon}(\by)\bra_{\Omega}}\cdot
	\begin{pmatrix}
	1 \\ \mathrm{i}
	\end{pmatrix}= -\frac12\overline{\bla \Phi_1^{(0)}(\by), \mathrm{i}\boldsymbol{p}^{\sigma}\Phi_2^{(0)}(\by)\bra_{\Omega}}\cdot\begin{pmatrix}
	1 \\ \mathrm{i}
	\end{pmatrix}+\mathcal{O}(|\epsilon|).
\end{align}
Here, we use the symmetry arguments which is much simpler than that of \cite{fefferman2012honeycomb}. The key is to compute \eqref{eigen:epsilon} and \eqref{id:velocity}. Since $V(\by)$ is rotational invariance, we continue to derive the following three parts:
\begin{align}\label{id:v01}
\nonumber	|\Omega|\hat{V}(0,0)= \bla e^{\mathrm{i}\gk\cdot\by},~V(\by)e^{\mathrm{i}\gk\cdot\by}\bra_{\Omega}=\bla e^{\mathrm{i}R\gk\cdot\by},~V(\by)e^{\mathrm{i}R\gk\cdot\by}\bra_{\Omega}=\bla e^{\mathrm{i}R^2\gk\cdot\by},~V(\by)e^{\mathrm{i}R^2\gk\cdot\by}\bra_{\Omega};\\
\nonumber	|\Omega|\hat{V}(0,1)= \bla e^{\mathrm{i}\gk\cdot\by},~V(\by)e^{\mathrm{i}R\gk\cdot\by}\bra_{\Omega}=\bla e^{\mathrm{i}R\gk\cdot\by},~V(\by)e^{\mathrm{i}R^2\gk\cdot\by}\bra_{\Omega}=\bla e^{\mathrm{i}R^2\gk\cdot\by},~V(\by)e^{\mathrm{i}\gk\cdot\by}\bra_{\Omega};\\
	|\Omega|\hat{V}(-1,0)= \bla e^{\mathrm{i}\gk\cdot\by},~V(\by)e^{\mathrm{i}R^2\gk\cdot\by}\bra_{\Omega}=\bla e^{\mathrm{i}R\gk\cdot\by},~V(\by)e^{\mathrm{i}\gk\cdot\by}\bra_{\Omega}=\bla e^{\mathrm{i}R^2\gk\cdot\by},~V(\by)e^{\mathrm{i}R\gk\cdot\by}\bra_{\Omega}.
\end{align}
And, one can also verify that $\hat{V}(-1,0)=\hat{V}(0,1)$, i.e.,
\begin{align}\label{id:gammabeta}
\int_{\Omega} e^{-\text{i}\gk\cdot\by}V(\by)e^{\text{i}R^2\gk\cdot\by}d\by = \int_{\Omega} e^{-\text{i}R\gk\cdot\by}V(\by)e^{\text{i}\gk\cdot\by}d\by = \int_{\Omega} e^{\text{i}R\gk\cdot\by}V(\by)e^{-\text{i}\gk\cdot\by}d\by.
\end{align}

Substituting \eqref{phi0} into \eqref{eigen:epsilon} and employing the calculations in \eqref{id:v01}, we immediately get
\begin{equation}
	E^{\epsilon}_\tau=E^{(0)}+\epsilon  (\hat{V}(0,0)-\hat{V}(0,1))+\mathcal{O}(\epsilon^2)
\end{equation}
and
\begin{equation}
	E^{\epsilon}_1=E^{(0)}+\epsilon  (\hat{V}(0,0)+2\hat{V}(0,1))+\mathcal{O}(\epsilon^2)
\end{equation}
As long as $\hat{V}(0,1)\neq0$, the three-fold $L^2_{\gk}-$eigenvalue, $E^{(0)}=|\gk|^{\sigma}$, splits into two distinct eigenvalues continuously dependent on $\epsilon$: one is a two-fold eigenvalue $E_D^{\epsilon}$ in $L^2_{\gk,\tau}\oplus L^2_{\gk,\bar{\tau}}$, and the other is a simple $L^2_{\gk,1}-$eigenvalue $E_{1}^{\epsilon}$. This proves assertions $(1)-(2)$ of Proposition \ref{thm:intersec}.

On the other hand, $\Phi_1^{(0)}(\by)=\Phi_{\tau}^{(0)}(\by)$ and $\Phi_2^{(0)}(\by)=\Phi_{\bar{\tau}}^{(0)}(\by)$, we first compute
\begin{equation}
\begin{split}
	-\bla \Phi_1^{(0)}(\by), \text{i}\boldsymbol{p}^{\sigma}\Phi_2^{(0)}(\by)\bra_{\Omega} =\frac13 \sigma\big|\gk\big|^{\sigma-2}\Big[\gk+\bar{\tau}R\gk+\tau R^2\gk\Big]=\frac{2\pi}3 \sigma\big|\gk\big|^{\sigma-2}
\begin{pmatrix}
1 \\
\text{i}
\end{pmatrix}.
\end{split}
\end{equation}
Substituting above into \eqref{id:velocity}, we immediately obtain \eqref{id:velocity1}. Thus we complete the proof.

\end{proof}

\end{appendices}

\bibliographystyle{plain}

\bibliography{frac}

\begin{thebibliography}{10}

\bibitem{ablowitz2010Evolution}
M.~J. Ablowitz and Y.~Zhu.
\newblock Evolution of {B}loch-mode envelopes in two-dimensional generalized
  honeycomb lattices.
\newblock {\em Phys. Rev. A.}, 82(1):131--133, 2010.

\bibitem{ablowitz2013Nonlinear2}
M.~J. Ablowitz and Y.~Zhu.
\newblock Nonlinear wave packets in deformed honeycomb lattices.
\newblock {\em SIAM J. Appl. Math.}, 73(6):1959--1979, 2013.

\bibitem{allaire2011diffractive}
G.~Allaire, M.~Palombaro, and J.~Rauch.
\newblock Diffractive geometric optics for {B}loch wave packets.
\newblock {\em Arch. Ration. Mech. Anal.}, 202(2):373--426, 2011.

\bibitem{ammari2018honeycomb}
H.~Ammari, B.~Fitzpatrick, H.~Lee, E.~O. Hiltunen, and S.~Yu.
\newblock Honeycomb-lattice {M}innaert bubbles.
\newblock {\em arXiv preprint arXiv:1811.03905}, 2018.

\bibitem{ammari2018high}
H.~Ammari, E.~O. Hiltunen, and S.~Yu.
\newblock A high-frequency homogenization approach near the {D}irac points in
  bubbly honeycomb crystals.
\newblock {\em arXiv preprint arXiv:1812.06178}, 2018.

\bibitem{arbunich2018rigorous}
J.~Arbunich and C.~Sparber.
\newblock Rigorous derivation of nonlinear {D}irac equations for wave
  propagation in honeycomb structures.
\newblock {\em J. Math. Phys.}, 59(1):011509, 2018.

\bibitem{bal2017topological}
G.~Bal.
\newblock Topological protection of perturbed edge states.
\newblock {\em arXiv preprint arXiv:1709.00605}, 2017.

\bibitem{bal2019continuous}
G.~Bal.
\newblock Continuous bulk and interface description of topological insulators.
\newblock {\em J. Math. Phys.}, 60(8):081506, 2019.

\bibitem{bensoussan1978asymptotic}
A.~Bensoussan, J.~L. Lions, and G.~Papanicolaou.
\newblock {\em Asymptotic analysis for periodic structures}.
\newblock North-Holland Pub. Co, 1978.

\bibitem{berkolaiko2018symmetry}
G.~Berkolaiko and A.~Comech.
\newblock Symmetry and {D}irac points in graphene spectrum.
\newblock {\em J. Spectr. Theory}, 8(3):1099--1148, 2018.

\bibitem{caffarelli2007extension}
L.~Caffarelli and L.~Silvestre.
\newblock An extension problem related to the fractional {L}aplacian.
\newblock {\em Comm. partial differential equations}, 32(8):1245--1260, 2007.

\bibitem{davila2014concentrating}
J.~D{\'a}vila, M.~Del~Pino, and J.~Wei.
\newblock Concentrating standing waves for the fractional nonlinear
  {S}chr{\"o}dinger equation.
\newblock {\em J. Differential Equations}, 256(2):858--892, 2014.

\bibitem{drouot2020edge}
A.~Drouot and M.~I. Weinstein.
\newblock Edge states and the valley {H}all effect.
\newblock {\em Adv. Math.}, 368:107142, 2020.

\bibitem{e2013asymptotic}
W.~E, J.~Lu, and X.~Yang.
\newblock Asymptotic analysis of quantum dynamics in crystals: the
  {B}loch-{W}igner transform, {B}loch dynamics and {B}erry phase.
\newblock {\em Acta Math. Appl. Sin. Engl. Ser.}, 29(3):465--476, 2013.

\bibitem{fefferman2012honeycomb}
C.~Fefferman and M.~I. Weinstein.
\newblock Honeycomb lattice potentials and {D}irac points.
\newblock {\em J. Amer. Math. Soc.}, 25(4):1169--1220, 2012.

\bibitem{fefferman2017topologically}
C.~L. Fefferman, J.~P. Lee-Thorp, and M.~I. Weinstein.
\newblock {\em Topologically protected states in one-dimensional systems},
  volume 247.
\newblock Mem. Amer. Math. Soc., 2017.

\bibitem{fefferman2014wave}
C.~L. Fefferman and M.~I. Weinstein.
\newblock Wave packets in honeycomb structures and two-dimensional {D}irac
  equations.
\newblock {\em Comm. Math. Phys.}, 326(1):251--286, 2014.

\bibitem{geim2007The}
A.~K. Geim and K.~S. Novoselov.
\newblock The rise of graphene.
\newblock {\em Nature Materials}, 6(3):183--91, 2007.

\bibitem{guo2019bloch}
H.~Guo, X.~Yang, and Y.~Zhu.
\newblock Bloch theory-based gradient recovery method for computing topological
  edge modes in photonic graphene.
\newblock {\em J. Comput. Phys.}, 379:403--420, 2019.

\bibitem{haldane2008Possible}
F.~D.~M. Haldane and S.~Raghu.
\newblock Possible realization of directional optical waveguides in photonic
  crystals with broken time-reversal symmetry.
\newblock {\em Phys. Rev. Lett.}, 100(1):013904, 2008.

\bibitem{hong2017new}
Y.~Hong and Y.~Sire.
\newblock A new class of traveling solitons for cubic fractional nonlinear
  {S}chr{\"o}dinger equations.
\newblock {\em Nonlinearity}, 30(4):1262, 2017.

\bibitem{hu2019linear}
P.~Hu, L.~Hong, and Y.~Zhu.
\newblock Linear and nonlinear electromagnetic waves in modulated honeycomb
  media.
\newblock {\em Stud. Appl. Math.}, 144(1):18--45, 2020.

\bibitem{jin2011mathematical}
S.~Jin, P.~Markowich, and C.~Sparber.
\newblock Mathematical and computational methods for semiclassical
  {S}chr{\"o}dinger equations.
\newblock {\em Acta Numer.}, 20:121--209, 2011.

\bibitem{joannopoulos2011photonic}
J.~Joannopoulos, S.~Johnson, J.~Winn, and R.~Meade.
\newblock {\em Photonic crystals: Molding the flow of light - second edition}.
\newblock Princeton University Press, 2011.

\bibitem{kato1975cauchy}
T.~Kato.
\newblock The {C}auchy problem for quasi-linear symmetric hyperbolic systems.
\newblock {\em Arch. Ration. Mech. Anal.}, 58(3):181--205, 1975.

\bibitem{keller2018spectral}
R.~Keller, J.~Marzuola, B.~Osting, and M.~I. Weinstein.
\newblock Spectral band degeneracies of $\frac{\pi}2$-rotationally invariant
  periodic {S}chr\"odinger operators.
\newblock {\em Multiscale Model. Simul.}, 16(4):1684--1731, 2018.

\bibitem{kuchment2012floquet}
P.~Kuchment.
\newblock {\em Floquet theory for partial differential equations}, volume~60.
\newblock Birkh{\"a}user, 2012.

\bibitem{laskin2000fractionallevy}
N.~Laskin.
\newblock Fractional quantum mechanics and {L}{\'e}vy path integrals.
\newblock {\em Phys. Lett. A}, 268(4-6):298--305, 2000.

\bibitem{lee2019elliptic}
J.~P. Lee-Thorp, M.~I. Weinstein, and Y.~Zhu.
\newblock Elliptic operators with honeycomb symmetry: {D}irac points, edge
  states and applications to photonic graphene.
\newblock {\em Arch. Ration. Mech. Anal.}, 232(1):1--63, 2019.

\bibitem{lemm2016holder}
M.~Lemm.
\newblock On the {H}{\"o}lder regularity for the fractional {S}chr{\"o}dinger
  equation and its improvement for radial data.
\newblock {\em Comm. Partial Differential Equations}, 41(11):1761--1792, 2016.

\bibitem{longhi2015fractional}
S.~Longhi.
\newblock Fractional {S}chr{\"o}dinger equation in optics.
\newblock {\em Opt. Lett.}, 40(6):1117--1120, 2015.

\bibitem{majda2012compressible}
A.~Majda.
\newblock {\em Compressible fluid flow and systems of conservation laws in
  several space variables}, volume~53.
\newblock Springer Science $\And$ Business Media, 2012.

\bibitem{pelinovsky2011localization}
D.~Pelinovsky.
\newblock {\em Localization in periodic potentials: from {S}chr{\"o}dinger
  operators to the {G}ross--{P}itaevskii equation}, volume 390.
\newblock Cambridge University Press, 2011.

\bibitem{racke1992lectures}
R.~Racke.
\newblock Lectures on nonlinear evolution equations.
\newblock {\em Initial value problems, Aspect of Mathematics E}, 19, 1992.

\bibitem{raghu2008analogs}
S.~Raghu and F.~D.~M. Haldane.
\newblock Analogs of quantum-{H}all-effect edge states in photonic crystals.
\newblock {\em Phys. Rev. A}, 78(3):033834, 2008.

\bibitem{rechtsman2013Photonic}
M.~C. Rechtsman, J.~M. Zeuner, Y.~Plotnik, Y.~Lumer, D.~Podolsky, F.~Dreisow,
  S.~Nolte, M.~Segev, and A.~Szameit.
\newblock Photonic {F}loquet topological insulators.
\newblock {\em Nature}, 496(7444):196--200, 2013.

\bibitem{reed1978methods}
M.~Reed and B.~Simon.
\newblock {\em Methods of modern mathematical physics, {IV}: {A}nalysis of
  operators}.
\newblock Academic press, 1978.

\bibitem{roncal2016fractional}
L.~Roncal and P.~Stinga.
\newblock Fractional {L}aplacian on the torus.
\newblock {\em Commun. Contemp. Math.}, 18(03):1550033, 2016.

\bibitem{secchi2013ground}
S.~Secchi.
\newblock Ground state solutions for nonlinear fractional {S}chr{\"o}dinger
  equations in $\mathbb{R}^{N}$.
\newblock {\em J. Math. Phys.}, 54(3):031501, 2013.

\bibitem{sulem2007nonlinear}
C.~Sulem and P.~Sulem.
\newblock {\em The nonlinear {S}chr{\"o}dinger equation: self-focusing and wave
  collapse}, volume 139.
\newblock Springer Science $\And$ Business Media, 2007.

\bibitem{xie2019wave}
P.~Xie and Y.~Zhu.
\newblock Wave packet dynamics in slowly modulated photonic graphene.
\newblock {\em J. Differential Equations}, 267(10):5775--5808, 2019.

\bibitem{yong1999singular}
W.-A. Yong.
\newblock Singular perturbations of first-order hyperbolic systems with stiff
  source terms.
\newblock {\em J. Differential Equations}, 155(1):89--132, 1999.

\bibitem{zhang2015propagation}
Y.~Zhang, X.~Liu, M.~Beli{\'c}, W.~Zhong, Y.~Zhang, and M.~Xiao.
\newblock Propagation dynamics of a light beam in a fractional
  {S}chr{\"o}dinger equation.
\newblock {\em Phys. Rev. Lett.}, 115(18):180403, 2015.

\end{thebibliography}

\end{document}